\documentclass[12pt,a4paper]{article}
\usepackage{amssymb,amsfonts,amsfonts,amsthm,amscd,amssymb,amstext,amsmath,graphicx,color}
\newtheorem{theorem}{Theorem}

\newtheorem{definition}[theorem]{Definition}

\newtheorem{lemma}[theorem]{Lemma}

\newtheorem{proposition}[theorem]{Proposition}
\newtheorem{remark}[theorem]{Remark}

\newtheorem{conjecture}[theorem]{Conjecture}

\title{Explicit examples in Ergodic Optimization}

\author{ Hermes H. Ferreira, Artur O. Lopes and Elismar R. Oliveira\\ Inst. Mat - UFRGS, Brazil \\ }

\begin{document}
\maketitle

\begin{abstract} Denote by $T$ the transformation $T(x)= 2 \,x $ (mod 1). Given a potential $A:S^1 \to \mathbb{R}$   we are  interested in exhibiting in several examples  the  explicit expression for the calibrated subaction $V: S^1 \to \mathbb{R}$ for $A$. The action of the $1/2$ iterative procedure  $\mathcal{G}$,  acting on continuous functions $f: S^1 \to \mathbb{R}$, was analyzed in a companion paper. Given an initial condition $f_0$, the sequence, $\mathcal{G}^n(f_0)$ will converge to a subaction. The sharp numerical evidence obtained from this iteration  allow us to guess  explicit expressions for the subaction in several worked examples: among them for  $A(x) = \sin^2  ( 2 \pi x)$ and $A(x) = \sin  ( 2 \pi x)$. Here, among other things, we present  piecewise analytical expressions for several calibrated subactions. The iterative procedure can  also be applied to the estimation of the joint spectral radius of matrices.
We also analyze the iteration of $\mathcal{G}$ when the subaction is not unique. Moreover, we briefly present the version
 of the $1/2$ iterative procedure for the estimation of the main eigenfunction of the Ruelle operator.
\end{abstract}

\section{Introduction} \label{gal}

Here we will present several examples in Ergodic Optimization where one can exhibit the maximizing probability and the subaction. The  $1/2$ iterative procedure is a tool (in some cases) for the corroboration of what is calculated or a helpful instrument to get important information. Comment about this last point: suppose someone in a specific example (not covered by the examples described in the present text) does not know the explicit expression for the maximizing probability and the subaction.  We want to show, through several worked examples, how one can proceed (using the  $1/2$ iterative procedure)  in order to try to get explicit information.	

Denote by $T:S^1 \to S^1$ the transformation
	$T(x)= 2\, x$ (mod 1).
We also denote by $\tau_1:[0,1) \to [0,1/2)$ and $\tau_2: [0,1) \to [1/2,1)$ the two inverse branches of $T$.	
\begin{definition}\label{max}  For a continuous function $A:S^1\to \mathbb{R}$ we denote the maximal ergodic  value  the number
$$m(A)= \sup_{ \rho\, {\rm \;is\; invariant\; for\;}\, T} \int A \, d \rho.$$
Any invariant probability $\mu$ which attains such 	supremum is called a {\bf  maximizing probability}
\end{definition}

For general  properties of maximizing probabilities see \cite{BLL}, \cite{G1}, \cite{J1}, \cite{CLT} and \cite{leplaideur-max}. A recent survey by O. Jenkinson (see \cite{J2})  covers the more recent literature on the topic.
We will assume here in most of the cases  that  $A$ is at least H\"older continuous. The results we consider here  can  also  be applied  to the case when $A$ acts
on the interval $[0,1]$ (non periodic setting).
	\begin{definition}\label{sub1}  The union of the supports of all the maximizing probabilities  is called the  Mather set for $A$.
\end{definition}
The maximizing probability does not have to be unique.

	\begin{definition}\label{sub} Given  $A:S^1\to \mathbb{R}$, then
		a continuous function $V: S^1 \to \mathbb{R} $ which satisfies for any $x\in S^1$:
		
		\begin{equation}\label{c} V(x)=\max_{T(y)=x} [A(y)+ V(y)-m(A)]\end{equation}
		is called a
		calibrated subaction for $A$.
	\end{definition}
From an explicit calibrated subaction one can guess where is the support of the maximizing probability (see important property below).  The subaction also provides important information for computing the
deviation function when temperature goes to zero in Thermodynamic Formalism (see \cite{BLT}); see also   \cite{BLL}, \cite{BLL1}, \cite{BGT}, \cite{GT}, \cite{Lep}, \cite{Bre}, \cite{BLM}, \cite{LM} for zero temperature limits.

Defining a  new function $R$  we get
 \begin{equation}\label{x1} R(x)\,:=\, V(T(x)) -  V(x) - A(x) + m(A) \geq 0.\end{equation}
One can show that for all points $x$ in the Mather set $R(x)=0.$ {\bf An important property is:}  if an invariant probability has support inside the set of points where $R=0$, then, this probability is maximizing (see \cite{CLT} or \cite{Bou1}). An interesting generic property related to the important property is described in \cite{GLT})

If the potential $A$ is H\"older and the maximizing probability is unique then the calibrated subaction is unique up to adding constants. Generically, in the H\"older class  the maximizing probability has support on a unique periodic orbit (see   \cite {Con}, \cite{CLT}). Similar properties  on the $C^0$ class are not true (see \cite{J1},  \cite{Bou2}, \cite{Mor} and \cite{Tal}).

Most of the questions in Ergodic Optimization are analyzed under certain premises: a) when  the potential $A$ is just continuous, and, b) when  it is assumed some regularity (as Lipschitz or H\"older) for $A$.
The two cases are conceptually distinct: in the first case, generically, the maximizing probability has support on the all space (see \cite{Bou2} and \cite{J2}) and in the second case, generically, the support has support on a periodic orbit (see \cite{Con}  and \cite{CLT}). In case a), generically, subactions are of no help. It is in case b)  that subactions are of great help for identifying the support of the maximizing probability. 

Given a H\"older potential $A$ we are interested in obtaining explicit expressions for the associated calibrated subaction $V$, and also for $m(A)$. We will do that with the help of the  $1/2$ iterative procedure described in the companion paper \cite{FLO}. In the case the maximizing probability is unique (a generic property) the iteration procedure will converge to the subaction $V$, the initial condition does not matter.

\begin{definition} In the set of continuous functions from $S^1$  to $\mathbb{R}$ we denote by $\sim$ the equivalence relation $f \sim g$, if $f-g$ is a constant.
The set of classes is denoted by $\mathcal{C}$ and, by convention, we will consider in each class a representative that has supremum equal to zero.
\end{definition}

\begin{definition}\label{op1}  Given  $A:S^1\to \mathbb{R}$ we consider the operator $\mathcal{G}= \mathcal{G}_A: \mathcal{C} \to \mathcal{C}$, such that, for $f: S^1 \to \mathbb{R} $, we have   $\mathcal{G}_A(f)=g$, if
		
$$\mathcal{G}_A(f)(x)=   g(x)=\frac{\max_{T(y)=x} [A(y)+  f(y)] + f(x)}{2} \,\,\,\,-$$
\begin{equation}\label{c1} \sup_{s\in S^1} \frac{\max_{T(r)=s} [A(r)+  f(r)] + f(r)}{2},
\end{equation}
for any $x\in S^1$. The procedure defined by the iteration $\mathcal{G}^n (f_0),$ $n \in \mathbb{N}$,  will be called the $1/2$  iterative procedure.
\end{definition}

It is known - in the case where the calibrated subaction is unique (up to adding constant) -  that
given any $f_0\in \mathcal{C}$, it will follow that  $\displaystyle\lim_{n \to \infty} \mathcal{G}^n (f_0)=u$,
where $u$ is the calibrated subaction  on the set $\mathcal{C}$  (see   \cite{FLO}).
This follows from results concerning a general type of iterative procedure (taking the advantage of the $1/2$ factor) discussed for instance in \cite{Ma1}, \cite{Dot2} or \cite{Ishi} (versions of this kind of result appeared before in the literature in different forms).  $\mathcal{G}$ is a weak contraction but not a strong contraction.

In the case where there is more than one maximizing probability more than one calibrated subaction may exist (see \cite{GL1}). In this case, there are different basins of attractions (see Seection \ref{more}) associated to different subactions (depending where one begins - the initial condition $f_0$ -  the iteration of the  $1/2$ iterative procedure). 

We point out the interesting papers \cite{Bou1} and \cite{AJ} where an estimation of the support of the maximizing probability is obtained for a certain class of potentials (but not using the approach described by \eqref{c1}).

Several of the pictures of the graphs of different subactions $V$ we will present here  were obtained by iterating the operator $\mathcal{G}$ applied to the initial function $f_0=0$. The  $1/2$ iterative procedure defined by the approximation of $V$ via  $\mathcal{G}^n (f_0)$ provides very sharp results and this will help us to get explicit examples of subactions. In some of the examples we consider the potentials $A(x) = \sin^2 ( 2 \pi x)$ (section \ref{ggu}) and $A(x) = \sin (2 \pi x)$ (section \ref{kk}).

The joint spectral radius is a generalization of the classical notion of spectral radius of a matrix, to sets of matrices.
The concept was introduced in 1960 by G-C. Rota and G. Strang. Several different kinds of algorithms were proposed for the joint spectral radius computability. In \cite{Conze} and in \cite{JP} the authors describe an interesting connection of this concept with Ergodic Optimization.
The analysis of the maximizing probability on the case of estimation of the spectral radius (which requires the calculus of  $m(A)$ for a certain potential $A$) will be considered here in section~\ref{jsp}.

We analyze the case where there is more than one calibrated subaction in section \ref{more}. Depending on the initial condition $f_0$ the iteration
$\mathcal{G}^n (f_0)$, when $n \to \infty$, may converge to different subactions. We also investigate the influence of the flatness of the potential on the flatness of the subaction.
In this section we just plot the graphs we get from the numerical iteration and we do not provide mathematical proofs.

We  consider in section~\ref{Cantor}  the case where
$A(x)=-d(x,K)$, and $\displaystyle d(x,K)=\min_{k\in K}|x-k|$ and $K\subset[0,1]$ is the Cantor.
We present some conjectures but we do  not provide mathematical proofs.  We believe is interesting for future work to know what one would expect in this case.

An example of a potential $A$ which is equal to its subaction $V$ is presented in section~\ref{equal}.

In section \ref{Ru} we will show how to adapt the  $1/2$ iterative procedure for estimating the main eigenfunction of the Ruelle operator.

In the Appendix \ref{app} we present the proof of some more technical results discussed before.

\begin{figure}[h]
  \centering
  \includegraphics[height=3cm, scale=0.5]{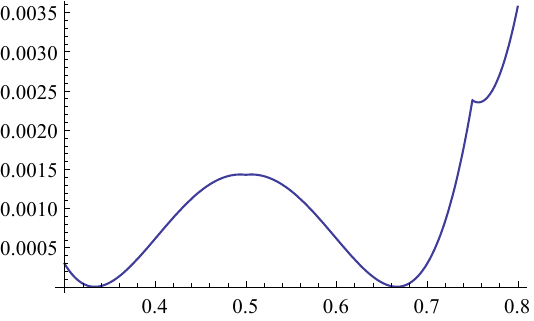}
  \caption{ Case $A(x) = - (x-1/3)^2 (x-2/3)^2 $ and $T(x)= 2 \, x $ (mod 1) - This picture shows the graph (plotted on Mathematica) of $R$ (see \eqref{x1}) obtained from an approximation of the calibrated subaction $u$ after $7$ iterations of the $1/2$ iterative procedure. We can infer from this figure (and the important property for $R$) that the maximizing probability has support on the periodic orbit $\{1/3,2/3\}$ as expected. Therefore, this iterative procedure  has the potential to display the support of the maximizing probability. }
  \label{fig:RR}
\end{figure}

As an example of the kind of result we can get with our methods we show in Figure \ref{fig:RR} (for the where case   $A(x) = - (x-1/3)^2 (x-2/3)^2 $ and $T(x)= 2 \, x $ (mod 1))  the graph of $R$ (obtained from the calibrated subaction $u$ we can get via the  $1/2$ iterative procedure).  Therefore, the  $1/2$ iterative procedure we will consider here can eventually exhibit  the support of  maximizing probabilities - via the function $R$ and the important property we mentioned before.

When the potential $A$ is analytic the subaction can sometimes be expressed as
\begin{equation}\label{eq:01}
V(x)= \sup \{V_1(x), V_2(x),..., V_r(x)\} ,\end{equation}
where $r>0$ and  $V_j$, $j\in\{1,2,..,r\}$, are analytic functions.

The number $r$ is equal to the period of the maximizing probability. It is of great significance to be able to estimate this number $r$ in order to get explicit solutions for $V$ (and, so to $m(A)$).  Expressions like  (\ref{eq:01}) are known to be true under the twist condition in several examples as described in the papers \cite{LOS} and \cite{LOT}. There the results were obtained via the use of the involution kernel and  techniques of Ergodic Transport.

Here most of the time the potential $A$ is of H\"older type.

We will follow next a certain general line of reasoning that will produce several explicit examples (see Section \ref{ees}). The main idea is: we assume some properties suggested by the graphs that we get  on the computer and then we develop some heuristic computations. In this way we are led to certain (piecewise) analytical expressions (each piecewise expression will be denoted by $V_j$) as in the above equation (\ref{eq:01}) for $V$. Then, finally, we check by hand if this expression for $V$ satisfies the calibrated subaction equation. We will elaborate on that: we will get recursive relations  of the form
\begin{equation} \label{ccp} V_j(x)+m(A)=V_{j+1}(\tau_i(x))+A(\tau_i(x)),\end{equation}
$i\in{1,2}, j=1,2,...r$, among the several $V_j$.

We explore these relations for deriving the candidates for being the different $V_j$.
Note that if one can get explicitly for  one of the subindices, let's say $j_0$, the expression for $V_{j_0}$, then, one can also get the others. In this way, we will get the final expression for  $V$, via expression (\ref{eq:01}). In this procedure, of course, we will also derive the value of $m(A).$ In some cases, the subaction has a series expression (see for instance (\ref{sisi})).

 From the historical point of view on the topic of Ergodic Optimization, it is needed to say that one of the first works on this subject was \cite{Conze} - a 1993 preprint that was not published. Among several results, the authors exhibit explicitly  the maximizing measure for the potential $A(x)=\sin^2(2 \pi x)$ (in Section \ref{ggu} we will achieve this result from an explicit expression of the subaction). The papers \cite{HO1} and \cite{HO2}  in turn had a more specific purpose: the analysis of optimal periodic orbits. The paper \cite{GS}  considered Markov chains with infinite states and asymptotically equilibrium measures (which are currently also  known as ground states) which are limits of Gibbs states when the temperature goes to zero. In \cite{Sav} the author considers Markov chains with finite symbols and a version of the subcohomology equation. The theory got more momentum when results similar to those on the Aubry-Mather theory (see \cite{GI}, \cite{Fathi} and \cite{Mane}) were obtained in a more systematic way. It is important to highlight a fundamental difference between these two theories: in the Aubry-Mather Theory, the convexity of the Lagrangean plays a fundamental role in the proofs of several results. The twist condition for the involution kernel  in some sense plays the role of convexity in Ergodic Optimization. The subaction (Definition \ref{sub}) corresponds in the Aubry-Mather theory to subsolutions of the Hamilton-Jacobi equation of Classical Mechanics.

The terminology Ergodic Optimization was established after the publication of the survey paper \cite{J1} by Oliver Jenkinson.

Explicit expressions for the subaction appeared previously in the literature in a few cases (for example in \cite{BLM}). Techniques of the max-plus algebra were used in section 7 in\cite{BLL} for this purpose.  Section 5 in \cite{LOT} presented several results  where the final expression was derived  from the associated involution kernel.
In \cite{BLL1} the expression was obtained via  techniques  related to the  Peierls’   barrier (see Lema 2.2).

\section{The case  $A(x)=-(x-\frac{1}{3})^2$} \label{pri}

Consider the potential $A(x)=-(x-\frac{1}{3})^2$.
We will present the explicit expression for $V$ on this case (which was not known before). Later we compare the explicit expression with the graph we get via the $1/2$ iterative procedure.

Note that such $A$ is not periodic on $[0,1]$. Therefore, we consider in this subsection that $T(x)= 2\,x$ (mod 1) acts on $[0,1]$. Consider also the inverse branches of $T$ given by $\tau_1(x)=\frac{x}{2}$ and $\tau_2(x)=\frac{x+1}{2}$. It is known from \cite{J3} that the maximizing probability in this case is Sturmian.

Looking Figure \ref{fig:my_label10} which we get from the  $1/2$ iterative procedure it is natural to assume the existence of $V_1, V_2,V_3,V_4$, such that
$$V_1(x)+m(A)=V_3\circ \tau_2(x) + A\circ \tau_2(x),\,V_2(x)+m(A)=V_1\circ \tau_1(x) + A\circ \tau_1(x),$$
\begin{equation} \label{olu} V_3(x)+m(A)=V_2\circ \tau_1(x) + A\circ \tau_1(x)\,\\,  V_4(x)+m(A)=V_3\circ \tau_1(x) + A\circ \tau_1(x).\end{equation}

The function $V_1$ is a continuation of $V_4$ when we look these functions $V_j$ as defined on $S^1$ (periodic). Equation (\ref{olu}) suggests that the maximizing probability has support  on an orbit of period three.
Note that $\frac{A(1/7)+A(2/7)+A(4/7)}{3}=-2/63$.

As $A$ is a polynomial of degree two is natural to try to express $V$ on the form
$V(x)=\sup \{\,V_i(x), i=1,2,3,4\, \}=\sup \{\,a_i+b_ix+c_ix^2,\, i=1,2,3,4\,\}$
for some choices of $a_i,b_i,c_i,$ $i=1,2,3,4$.
Assuming each $V_i(x) =\,a_i+b_ix+c_ix^2$ we can convert the four equations (\ref{olu}) in a linear system that can be easily solved.
From this procedure, we get \(m(A)=-2/63\). Moreover, we obtain
$V_1(x) =\frac{10}{63} - \frac{2x}{21}-\frac{x^2}{3}, V_2(x) =\frac{5}{63} + \frac{2x}{7}-\frac{x^2}{3}, V_3(x) =\frac{10x}{21}-\frac{x^2}{3},\text{ and } V_4(x) =-\frac{5}{63} + \frac{4x}{7}-\frac{x^2}{3}$.
A tedious calculation confirms that the $V$ we obtained from $ V(x)= \sup \{V_1(x), V_2(x),V_3(x), V_4(x)\} ,$
is really the calibrated subaction (with maximum value zero) for such $A$. In Figure \ref{fig:my_label11} we compare the graph of the approximated calibrated subaction obtained from the $1/2$ iterative procedure (in red) and the exact analytic expression for $V$ we obtained above (in blue). We have a perfect match. With 15 iterations of the $1/2$ iterative procedure, we get a good approximation of  $V$ (which was analytically obtained above) .

\begin{figure}[h] \center\includegraphics[, height=3cm,scale=0.35]{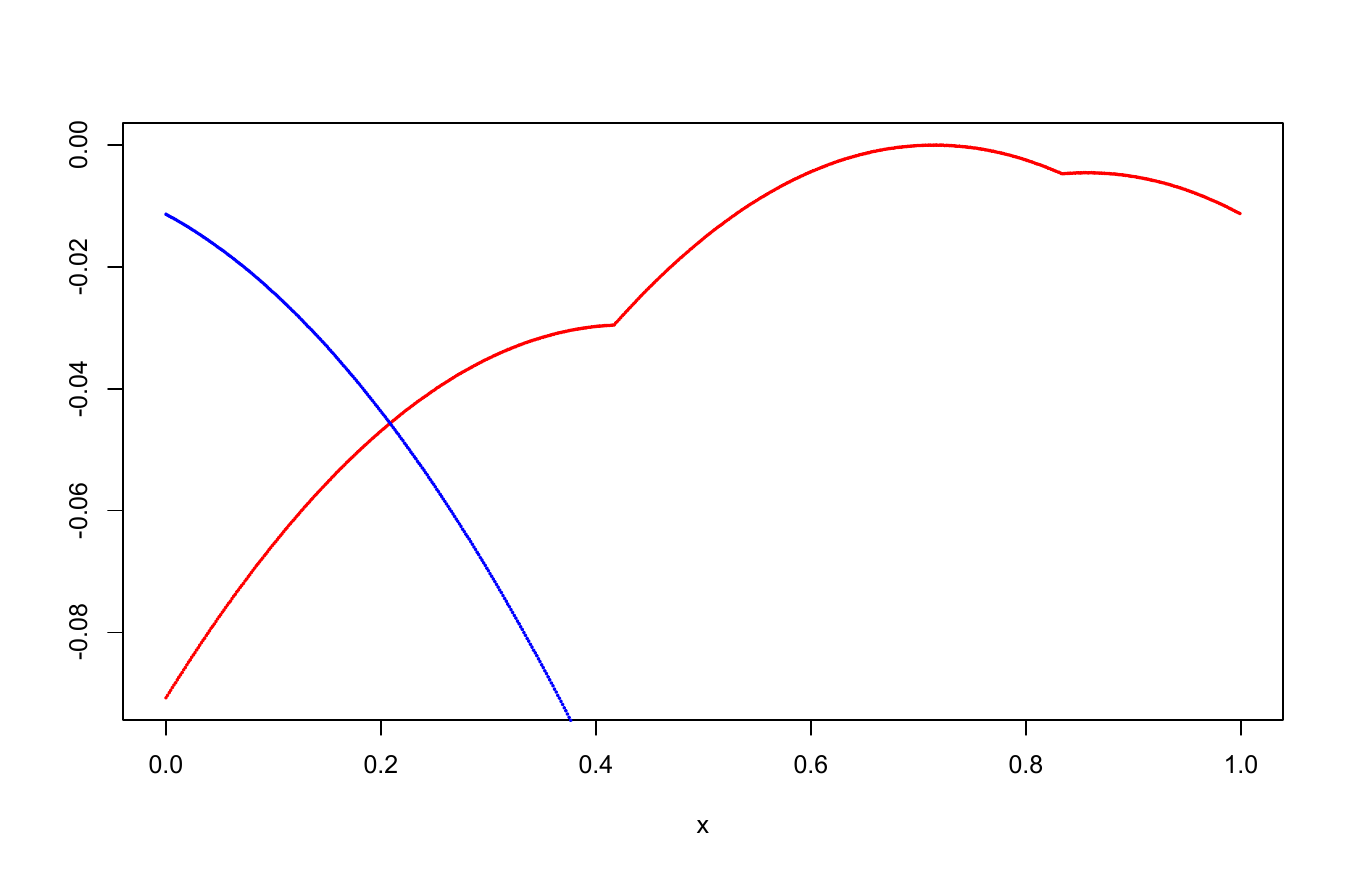}
	\caption{Case $A(x)=-(x-\frac{1}{3})^2 $ - The blue graph  describes the values of the approximation of the calibrated subaction $V$ where the $1/2$ iterative procedure detect that the realizer branch was $\tau_2$.  The red graph  describes the values of the approximation of $V$ where the $1/2$ iterative procedure detect that the realizer branch was $\tau_1$. The graph for the approximation of $V$ is the supremum of the two curves.  We iterate 15 times  $\mathcal{G}$ to get this picture.}
	\label{fig:my_label10}
\end{figure}
\section{An example for a weakly expanding system} \label{WKe}

This is an example where the exactly calibrated subaction $V$ is known. We will show that the $1/2$  iterative procedure performs fine in this case.

Consider $f:[0,1]\to [0,1]$, where
$$\left\{\begin{array}{l}
f(y)= \frac{y}{1-y},\mbox{ if },\, 0\leq y \leq \frac{1}{2},\\
f(y)= 2- \frac{1}{y},\mbox{ if }, \, \frac{1}{2}< y\leq 1,\\
\end{array}
\right.$$

\begin{figure}[h] 	
	\centering
	\includegraphics[height=3cm, scale=0.35]{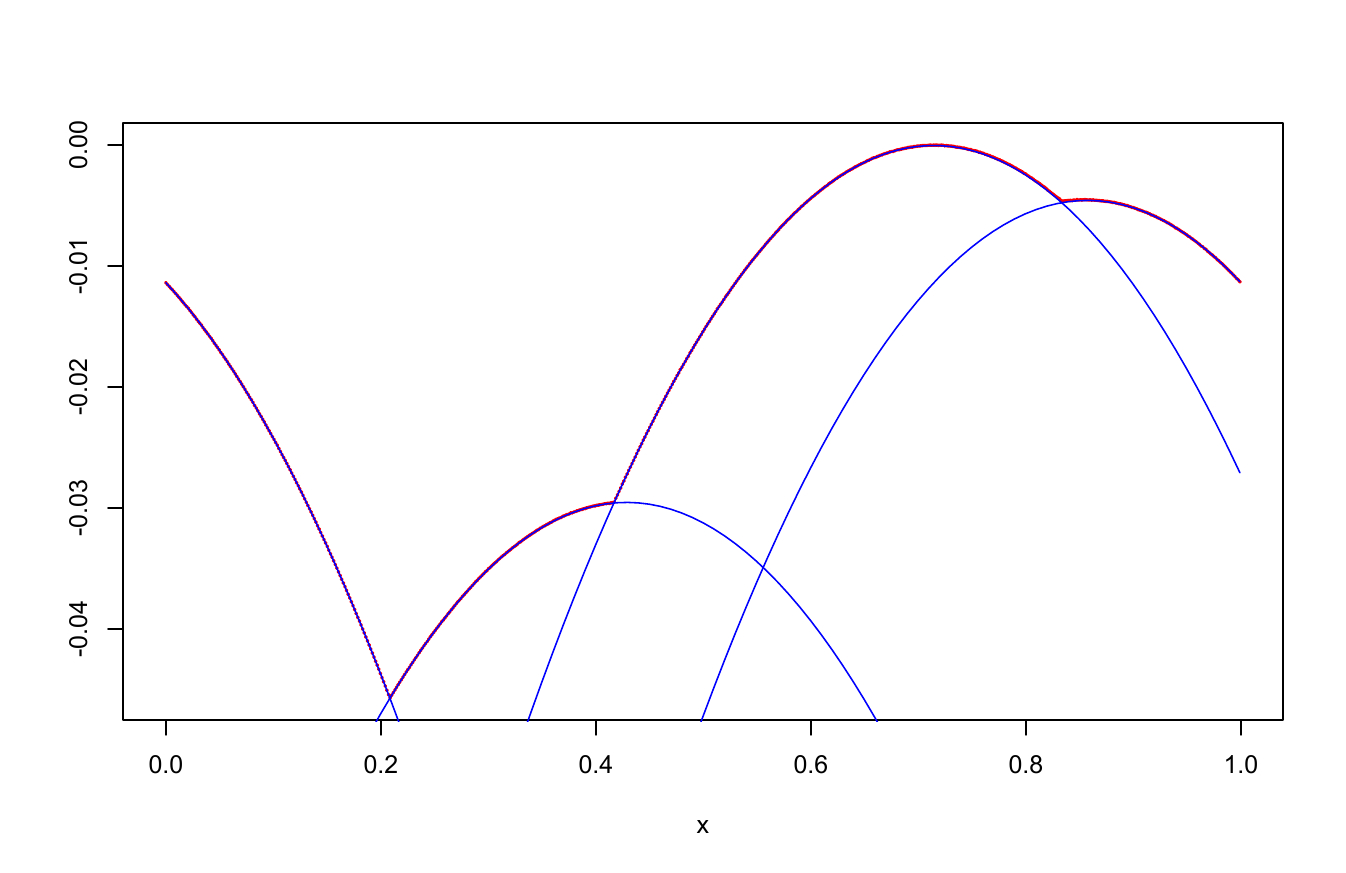}
	\caption{Case $A(x)=-(x-\frac{1}{3})^2 $ - In red we present the graph of the approximation of the calibrated subaction $V$ via the
		 $1/2$ iterative procedure. The picture in blue show the graphs of the different $V_j$, $j=1,2,3,4$.}
\label{fig:my_label11}
\end{figure}

and the potential $A(y) =  \log f '(y)$, where $f'$  is given by the
expression
$$\left\{\begin{array}{l}
f'(y)= \frac{1}{(1-y)^2},\mbox{ if },\, 0\leq y \leq \frac{1}{2},\\
f'(y)= \frac{1}{y^2},\mbox{ if }, \, \frac{1}{2}< y\leq 1.\\
\end{array}
\right.$$

The equation  for the calibrated subaction $V$ is
\begin{equation}\label{cf} V(x)=\max_{f(y)=x} [A(y)+ Vy)-m(A)].\end{equation}
We want to find the explicit calibrated subaction $V$ associated to  $A$ and also the value $m(A)$. The two inverse branches for $f$ are $\tau_1 (x) = \frac{x}{1+x}$ and
$\tau_2 (x) = \frac{1}{2-x}$. Consider $x_0= \frac{\sqrt{5}-1}{2}$ which is such that $f^2 (x_0)=x_0$ and $x_1=f(x_0).$
The maximizing probability for $A$ is $\frac{1}{2} ( \delta_{x_0} + \delta_{x_1})$.
Therefore,
$m(A)=\frac{1}{2}\left(A\left(\frac{\sqrt{5}-1}{2}\right)+A\left(\frac{\sqrt{5}-1}{\sqrt{5}+1}\right)\right)$.

Denote $F(y,x)$ the canonical natural extension of $f(y)$. The expression for the
transformation $F:[0,1]^2 \to [0,1]^2$ is described on the Appendix \ref{fixind}.

We say   that $W(x,y)$ is an involution kernel for $A(x)$  (see \cite{CLT}, \cite{LOS}, \cite{LOT}), if there is a $A^*(y)$, such that, for all $(y,x)$ we have:
$$A(F^{-1}(y,x))+    W( F^{-1}(y,x) )- W(y,x) =A^* (y). $$
We say that $A$ is symmetric if
$A^{*} (x) = A(x).$ This will be the case here.
The involution kernel for $A$ is $W(x,y)=2  \log(x + y - 2 x y)$ (see Appendix \ref{fixind}). Take
\begin{equation} \label{We} V(x) =  \sup \{ W(x_0,x), W(x_1,x)\}.
\end{equation}

\begin{figure}[h]
    \centering
    \includegraphics[height=5cm,scale=0.35]{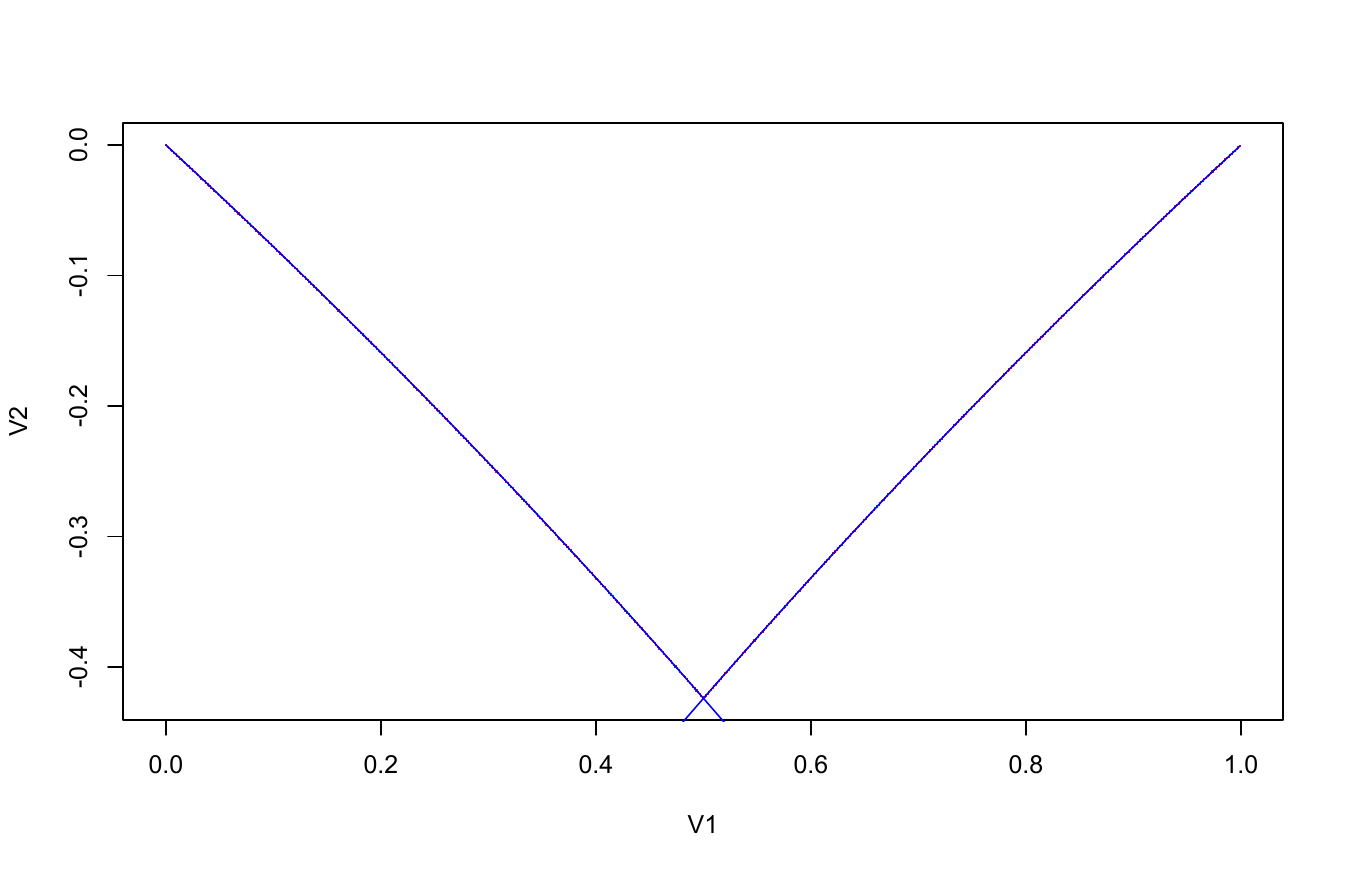}
    \caption{In red the graph of the  approximation of the calibrated subaction  we get from the  $1/2$ iterative procedure and in blue the two graphs of, respectively, $x \to W_A(x_0,x),$ and $x \to W_A(x_1,x)$.  The exact calibrated subaction $V$ (obtained analytically)  is $ V(x) =  \sup \{ W_A(x_0,x), W_A(x_1,x) \}$. The graph in red obliterates the ones in blue in a big part of the picture.}
    \label{fig:log_graf_analitico}
\end{figure}

One can show (a simple computation)  that such $V$ is the calibrated subaction for $A$.
In several examples the calibrated subaction has this form  (\ref{We})  (see  example 5 in pages 366-367 in \cite{LOT}).

Figure \ref{fig:log_graf_analitico} shows in red the graph of the  approximation we get from the  $1/2$ iterative procedure and in blue the two graphs of, respectively, $x \to W_A(x_0,x),$ and $x \to W_A(x_1,x)$.
This involution kernel is twist (see \cite{LOT} and \cite{LOS} for properties), that is, $\frac{\partial^2 W_A(x,y)}{\partial_x\, \partial_y}\leq 0$. When the potential $A$ is such that the associated involution kernel is twist some special properties can be obtained. This property replaces in some sense the convexity property which is essential in Aubry-Mather Theory.

The numerical values we get are  $x_0=0.3819...$, $x_1=0.6180...$, and $m(A) = \frac{ 1}{2}(A(x_0)+ A(x_1))= 0.9624...$.
We point out that if one considers instead  the potential $A(x) = - \log(f')$ then its associated
involution kernel is  $W(x,y)= - 2  \log(x + y - 2 x y)$.
In this case $m(A)=0.$ The maximizing probability $\mu$  has support on the set $\{0,1\}$.
This means that the support  of  $\mu$ is the union of two fixed points: $p_0=0$ and $p_1=1$ (when we consider that $f$ acts on $[0,1]$).
One can show that in a similar way as before the calibrated subaction $V$ is given by
\begin{equation} \label{Wee} V(x) =  \sup \{ W(p_0,x), W_A(p_1,x)\}.
\end{equation}

\section{A procedure to get piecewise analytic expressions} \label{ees}

In some examples we have to proceed in a different way from the previous one. We  will look for a way to express such initial $V_j$ via the relation
\begin{equation}\label{eq:1}
   V_j(x)-V_j(\eta(x))=F(x)-K,
\end{equation}
where $F$ and $\eta$ are functions and $K= N\,m(A)$, where $N$ is the period of the  maximizing orbit, $j=1,2,...,N$. In our examples $N=r$ ($S^1$ point of view), or, $r=N+1$ ($[0,1]$ point of view). The function $F$ will be chosen according to convenience in each kind of example. The value $K$ is a fixed variable on the process of trying to find the
calibrated subaction. We use the notation $\hat{m}(A) = \frac{K}{r}$ to express the fact  that we do not know  a priori the exact value $m(A)$ but in the end we will show that
 $m(A)= \hat{m}(A).$ We point out that  from \cite{BCLMS} we have the following property:  given $A$ and $V$, if we know   that for some constant $c$
\begin{equation} \label{fla} V(x)=\max_{T(y)=x} [A(y)+ V(y)-c],
\end{equation}
then, $V$ is a calibrated subaction and $c=m(A).$
We assume $\eta:[0,1]\rightarrow[0,1]$ is such that \begin{equation*}\eta^n:=\underbrace{\eta\circ\eta\circ...\circ\eta}_{n-times}\end{equation*}
satisfies $\displaystyle\lim_{n\rightarrow \infty}\eta^n(x)=q$ for some fixed point $q\in[0,1]$.
This indeed will happen in some of the examples we will consider.
Note that (\ref{eq:1}) implies
\begin{equation}\label{eq:2}
    V_j\circ \eta(x)-V_j\circ \eta^2(x)=F\circ \eta(x) - K.
\end{equation}
If $q$ is fixed by $\eta$ we get $F(q)=K$.
Therefore, adding (\ref{eq:1}) e (\ref{eq:2}) we get
\begin{equation}\label{eq:3}
    V_j(x)-V_j\circ \eta^2(x)=F(x)+F\circ \eta(x) - 2K.
\end{equation}
We can go on and inductively obtain for each $n$ in $ \mathbb{N}$,
\begin{equation}\label{eq:4}
    V_j(x)-V_j\circ \eta^{n}(x)=\sum_{i=0}^{n-1}\left[ F\circ \eta^i(x) - \,K\right] .
\end{equation}
If $V_j$ is continuous we get
$\displaystyle \lim_{n\rightarrow \infty}V_j\circ \eta^n(x)=V_j(q)$. Using the notation $\eta^0(x)=x$ we obtain finally a series (which should be the expression of this $V_j$ we are looking for)
\begin{equation}\label{eq:5}
V_j(x)=\lim_{n \rightarrow \infty}\sum_{i=0}^{n-1}\left[ F\circ \eta^i(x)  - K \right]- V_j(q) .\end{equation}

We can consider the truncated approximation
\begin{equation}\label{eq:6}
    V_j^{n*}(x)=\sum_{i=0}^{n-1}[F\circ \eta^i(x) - \,K] - V_j(q).
\end{equation}

We can assume that $V(q)=0$.
In this way each $V_j$ should be  given by
 \begin{equation} V_j(x)=\lim_{n \rightarrow \infty}   V_j^{n*}(x)=  \lim_{n \rightarrow \infty}\sum_{i=0}^{n-1}  (F\circ \eta^i(x) - n\,K) = \label{rty}  \sum_{i=0}^{\infty}  (F\circ \eta^i(x) - \,K)  ,\end{equation}
$j=1,2,...,r$, where $r$ is the number of $V_j$. All this is dependent of the smart choices of $F$ and $\eta$. In each example we have to show that the above limits  $V_j, $ $j=1,2,...,r$,  indeed exist. Moreover, we have to show   that
\begin{equation}
V(x)= \sup \{V_1(x), V_2(x),V_3(x),..., V_r(x)\} ,\end{equation}
solves the the subaction equation (\ref{c}) for $A$. When  $F$ is analytic  (if $A$ is analytic this will be the case in most of our examples)    the expression (\ref{rty}) will provide
an analytic expression for $V_j,$ $j=1,2...,r.$ In this case $V$ will be piecewise analytic. More than that, in most of the cases, there is an analytic dependence of
$F$ on the analytic potential $A$ (see Remark~\ref{oldremark3}). Under appropriate conditions (on absolutely convergence, etc.)
this will provide an analytic dependence of the calibrated subaction $V(x)$ for $A$, in each point $x$, on the potential $A$. In the computational procedure to be followed for getting such $V_j$ one does not know in advance the value $m(A)$. When $F$ has Lipschitz constant equal $M$ we get the estimate $ |F\circ \eta^i(x)-F(q)| \leq M | \eta^i(x)-q|$. In some of the examples we will get uniform convergence because
$\sum_{i=0}^{+ \infty} M | \eta^i(x)-q|$
is uniformly bounded. In this way the series defining  $V_j$ converges uniformly. We will follow the above reasoning in several examples to be presented next.


\section{The case  $A(x)=\sin^2(2 \pi x)$} \label{ggu}

Consider the periodic function
\(A(x)=\sin^2(2 \pi x)\),
\(T(x)=2\,x\,\text{mod}(1)\),
\(\tau_1(x)=\frac{x}{2}\),
\(\tau_2(x)=\frac{x+1}{2}\).
According to page 23 in \cite{Conze} the maximizing probability $\mu$ has support on the periodic orbit of period $2$ (the points $1/3$ and $2/3$). Therefore, we know that $m(A)= \frac{1}{2} (A(1/3) + A(2/3))\approx 0.75.$

In the graphs presented in Figure  \ref{fig:my_label} -  which were obtained from the $1/2$ iterative procedure we call
$V_2$ (blue color) the function we get when the realizer is on the branch $\tau_2$ and   $V_1$ (red color) the function we get when the realizer is on the branch $\tau_1$. The numerical result we get from the iterative procedure  shows the evidence (see Figure \ref{fig:my_label}) that  the  calibrated subaction $V$ should satisfy \begin{equation} \label{iopi} V(x)=\sup\,\{\,V_1(x),V_2(x)\,\}.\end{equation}
\begin{figure}[h]
    \centering
    \includegraphics[height=3cm, scale=0.3]{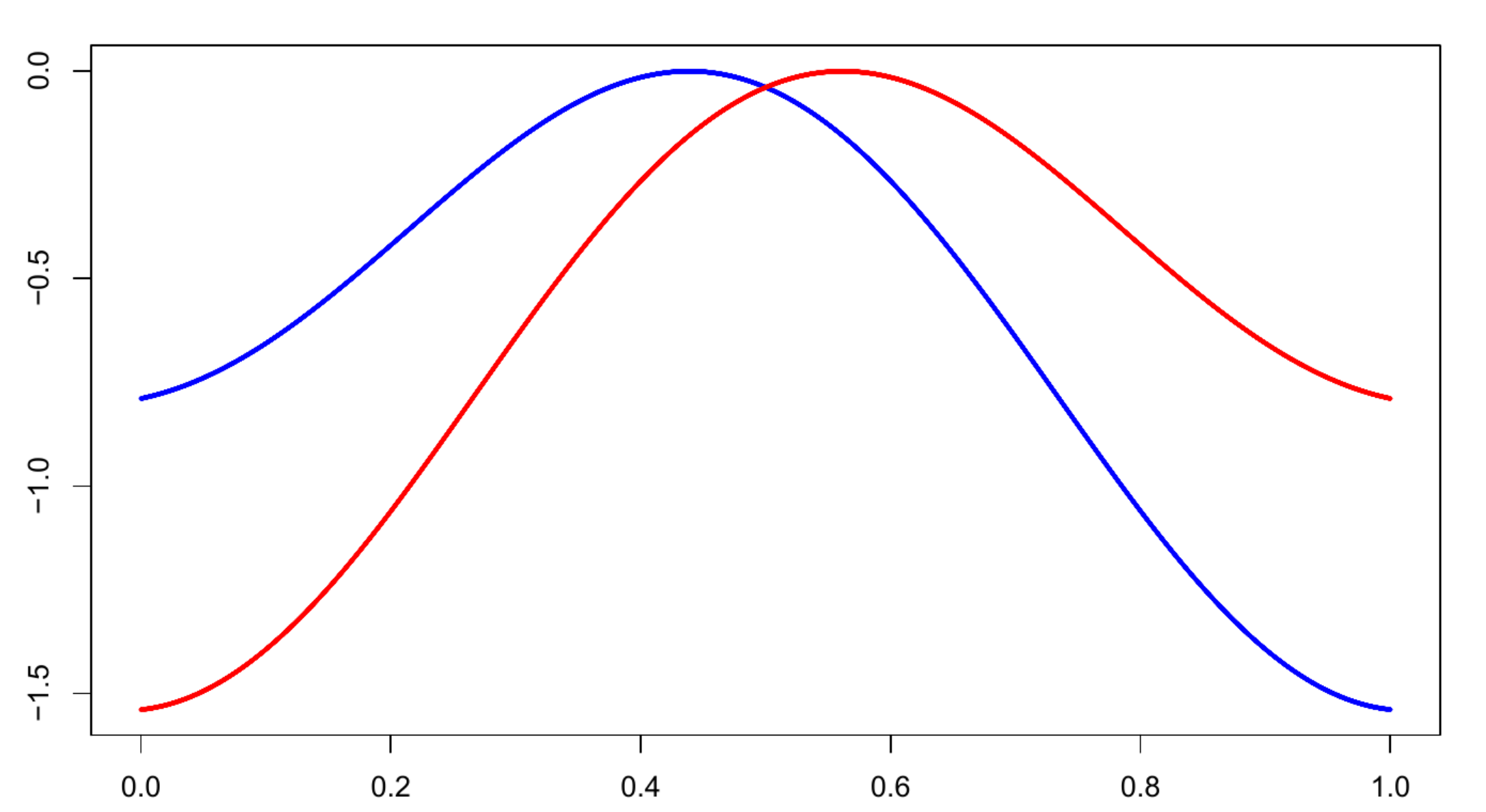}
    \caption{Case $\sin^2\,(2 \pi x)$ - From the $1/2$ iterative procedure  taking $\mathcal{G}^{20} (0)$  we get that the approximated subaction $V$   is  given by the supremum of the two functions in red and in blue. The graph in blue describe the values where the calibrated subaction equation is realized by the action of  $\tau_2$.  The graph in red describe the values where the calibrated subaction equation is realized by the action of  $\tau_1$.}
\label{fig:my_label}
\end{figure}
We will present an analytic expression for $V_2$.
We will show that
$$V_2(x) = \sum_{i=0}^{+\infty}\left[\,\,\sin^2\left(\pi\left(  \frac{2}{3} + \left(-\frac{1}{2}\right)^i(x-2/3) \right)\right)\,-\, \sin^2 (2 \pi/3) \,\,\right].$$
A power series expansion of $V_2$ around $2/3$ is presented in (\ref{power}). The expression for $V_1$ will follow from $V_1= V_2(1-x)$:
$$V_1(x)=  \sum_{i=0}^{+\infty} \left[\,\,\sin^2\left(\pi\left(  \frac{2}{3} - \left(-\frac{1}{2}\right)^i(1/3-x) \right)\right)-\, \sin^2 (2 \pi/3) \,\,\right].$$
This will finally  produce from (\ref{iopi}) the explicit expression for the subaction $V$ for such $A$. The proof that the above $V_1$ and $V_2$ are such that
$ V(x)=\sup\,\{\,V_1(x),V_2(x)\,\}$
is a calibrated subaction for $A$ will be done in Theorem~\ref{tete}. It follows from \cite{GL1} that in the case the potential $A$ has a symmetry of the kind $A(x)=A(1-x)$, then, the same property is true for the corresponding calibrated subaction. Then, $V_1(x)= V_2( 1-x).$ From this one can show that
$V_1(x)=V_2((x+1)/2)+A((x+1)/2)-\hat{m}(A).$ It is instructive to explain step by step our reasoning. The procedure can be applied to other examples.
\begin{figure}[h]
    \centering
    \includegraphics[height=3cm,scale=0.35]{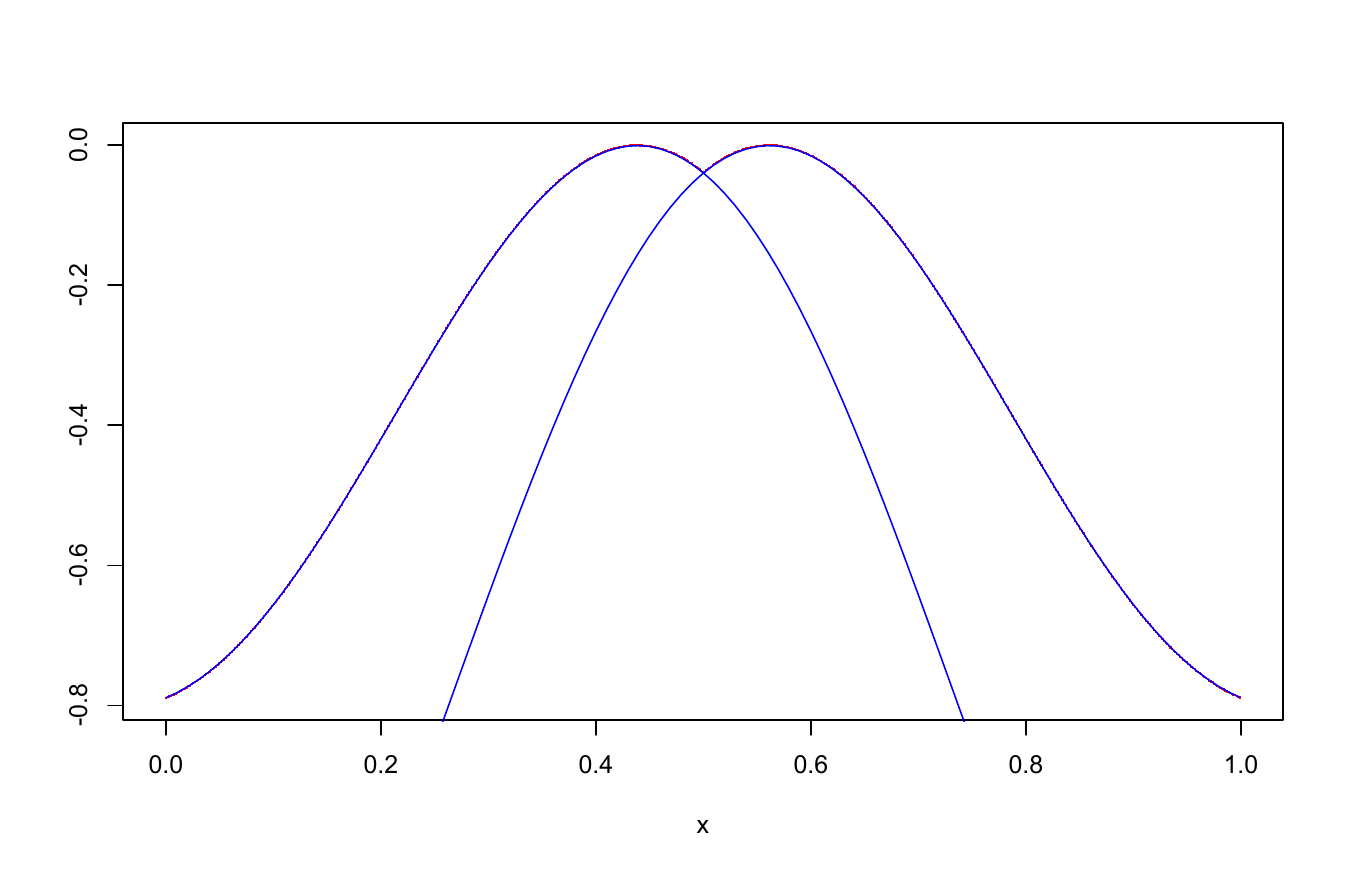}
    \caption{Case $\sin^2\,(2 \pi x)$ - The graph in red shows the numerical approximation of the subaction $V$ by $\mathcal{G}^{30} (0)$ with a discretization of $10^4$ points of the form $\frac{n}{10^4}$. In blue we show the graph of  $V_1^{4*}(x)$ e $V_2^{4*}(x)$
    (which approximate $V_1$ e $V_2$) according to (\ref{sin2_eta_series}).}
\label{gi3}
   \end{figure}
Following (\ref{ccp}) we assume the relation
\begin{equation}  \label{op1} V_2(x)+\hat{m}(A)=V_1(\tau_1(x))+A(\tau_1(x)),\end{equation}
and, also
\begin{equation}V_1(x)+\hat{m}(A)=V_2(\tau_2(x))+A(\tau_2(x)).\end{equation}

Therefore, by substitution
\begin{equation}\label{op2}   V_2(x)-V_2\left(\frac{x}{4}+\frac{1}{2}\right) = A\left(\frac{x}{2}\right)+ A\left(\frac{x}{4}+\frac{1}{2}\right)-2\,\hat{m}(A).\end{equation}
Taking \(\eta(x)=\frac{x}{4}+\frac{1}{2}\) and \(K=2\hat{m}(A)\), note that if $x \in [0,1]$, then
$\displaystyle\lim_{n\rightarrow +\infty}\eta^n(x)=2/3.$
Define
$F(x)= A(\frac{x}{2})+ A(\frac{x}{4}+\frac{1}{2}),$
then, by (\ref{eq:6})
$\displaystyle\lim_{n\rightarrow +\infty}F(\eta^n(x))=F(2/3)=A(1/3) + A(2/3)=2\hat{m}(A)$. We point out  that in the present case we already know from \cite{Conze} that the above $\hat{m}(A)=m(A).$
Note that $  F(x)=\sin^2(\pi x) + \sin^2(\pi x/2) $ is analytic.
\begin{remark}\label{oldremark3}
We point out that if we were considering another potential $A$ close to $\sin^2(2 \, \pi\, x)$, then, the reasoning we are going to consider below would apply in a similar way.  Note that $F$ depends nicely on $A$. In this case  (\ref{rty})  provides an analytical dependence of  $V$ on the nearby potential $A$.	
\end{remark}
The  $1/2$ iterative procedure produces the numerical approximation $m(A)\approx 0.75$.
We assume that $V(2/3)=0.$ Now we will express  $V_2$ - using (\ref{rty})  - up to constant via truncation
\begin{equation}\label{sin2_eta_series}
    V_2^{n*}(x)=\sum_{i=0}^{n-1}\,[\,F\circ \eta^i(x) - 2 m(A)\,].
\end{equation}
Figure \ref{gi3}  shows that for small values of $n$ one can get a good approximation of the subaction via $V_2^{n*}(x)$, $x\in [0,1]$.

\begin{proposition}\label{prop1} $\lim_{n \rightarrow +\infty}V_2^{n*}(x) $, $n \in \mathbb{N},$  given by (\ref{sin2_eta_series}), converges uniformly.
\end{proposition}
\begin{proof}
 We get $2\hat{m}(A)= \sin^2(2\pi/3)+\sin^2(\pi/3)$
and
\[|F \circ \eta^i(x)-2\hat{m}(A)|=\left |(\sin^2(\eta^i(x)\pi)-\sin^2(2\pi/3))+(\sin^2(\eta^i(x)\pi/2)-\sin^2(\pi/3)) \right |\]
\[\leq |\sin^2(\eta^i(x)\pi)-\sin^2(2\pi/3)| + |\sin^2(\eta^i(x)\pi/2)-\sin^2(\pi/3)|.\]
Moreover,
$\eta^i(x)=2/3\left( 1- \left(\frac{1}{4} \right)^i \right)+\frac{x}{4^i},$
which means
$\eta^i(x)-2/3= \frac{1}{4^i}\left(  x- 2/3\right).$
$\sin^2$ is  Lipschitz in $[0,1] $ for some constant $K$. Therefore,
$|\sin(x)-\sin(y)|\leq K|x-y|.$
Then,  $|\sin^2(\eta^i(x)\pi)-\sin^2(2\pi/3)| \leq |\eta^i(x)\pi-2\pi/3|= \frac{K\pi}{4^i}|x-2/3|$ and
$|\sin^2(\eta^i(x)\pi/2)-\sin^2(\pi/3)|\leq |\eta^i(x)\pi-\pi/3| \leq \frac{K\pi}{2}\frac{1}{4^i}|x-2/3|.$
From this
\[|\sum_{i=0}^{+\infty}(F\circ \eta^i(x)-2\hat{m}(A))| \leq \sum_{i=0}^{+\infty}|F\circ \eta^i(x)-2\hat{m}(A)|\]
\[\leq \sum_{i=0}^{+\infty}\left(\frac{K\pi}{4^i}|x-2/3| +\frac{K\pi}{2}\frac{1}{4^i}|x-2/3| \right)
\leq K\pi \sum_{i=0}^{+\infty}\frac{1}{4^i} < +\infty  .\]
\end{proof}

{\color{blue} .}
\begin{figure}
    \centering
    \includegraphics[height=3cm,scale=0.35]{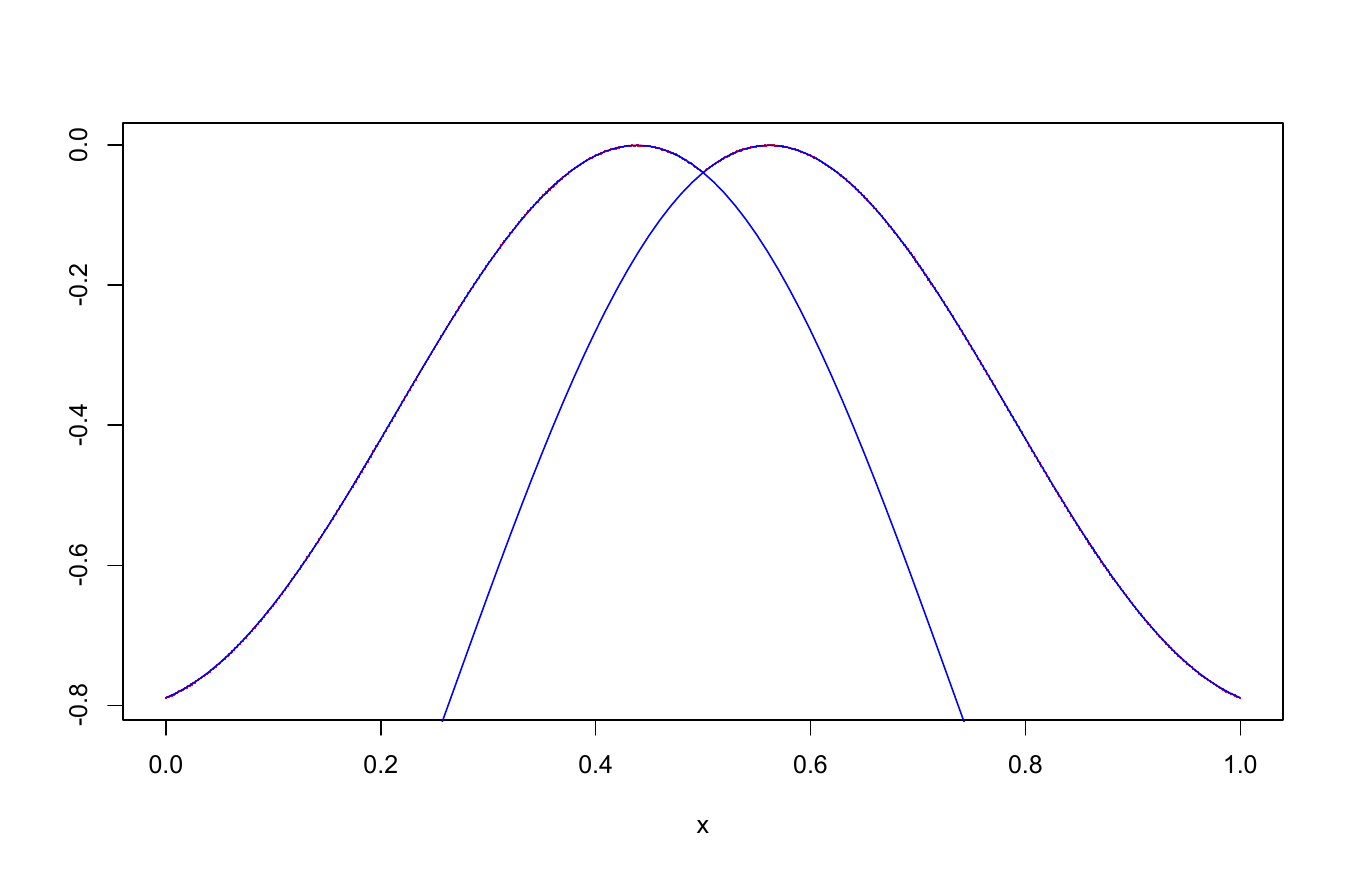}
    \caption{Case $\sin^2\,(2 \pi x)$ -  In red we show the graph of the approximation of the subaction $ V$ via the $1/2$  iterative procedure and in blue we show the result we get for  $V(x)=\sup \,\{V_1(x),V_2(x)\}$ via power series expansion truncated at order $10.$  That is, $V_2$ expressed by (\ref{khj})  and $V_1$ also in power expansion. As expected, they are virtually indistinguishable since the red portion agrees with the superior envelope of the blue curves.}
 \label{fig:my_label2}
\end{figure}

Denote \(\delta(x)=1-x/2\).
It is possible to get from the system  (\ref{op2}) that $V_1(x)=V_2(1-x)$ and $V_2(x)+m(A)=V_1(x/2)+A(x/2)$ we  get
$V_2(x)+m(A)=V_2(1-x/2)+A(x/2).$
As  $m(A)=A(2/3)$ and $\displaystyle\lim_{n\rightarrow +\infty}\delta^n(x)=2/3 $,  for
 $x \in  [0,1]$ we obtain
$V_2(x)-V_2(\delta(x))=A(x/2)-A(2/3)$.
From this we get
$V_2(x)-V_2(2/3)=\sum_{i=0}^{+\infty}\left( A(\delta^i(x)/2)-A(2/3) \right)$.
As \(V_2(2/3)= 0\), it follows that
$V_2(x)=\sum_{i=0}^{+\infty}\left( A(\delta^i(x)/2)-A(2/3) \right)$.
Finally, as $\delta^n(x+2/3)=\frac{2}{3}+\left( -\frac{1}{2}\right)^n\,x\,$, we obtain the expression
\begin{equation}\label{eq:sin2v}
V_2(x)=\sum_{i=0}^{+\infty}\left(\sin^2\left(\pi\left(  \frac{2}{3} + \left(-\frac{1}{2}\right)^i(x-2/3) \right)\right) -\sin^2(2\pi/3) \right).\end{equation}
The corresponding expression for $V_1$ can be obtained from the equality $V_1(x)= V_2( 1-x).$
\begin{figure}[h]
    \centering
    \includegraphics[scale=0.65]{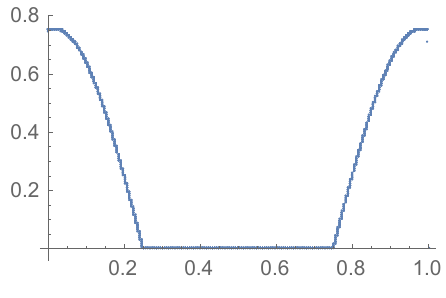}
    \caption{case $\sin^2(2 \pi x)$ - The graph of $R$  using the  approximation of the calibrated subaction. The orbit of period 2 is inside the set $R=0$.}
    \label{fig:RR7}
\end{figure}
We will show in the Appendix \ref{ggs}
that $ V(x)=\sup\,\{\,V_1(x),V_2(x)\,\} $
is a calibrated subaction for $A$.
Moreover, in the Appendix \ref{ggs} we will present a power series expansion around $2/3$  for $V_2$:
$$V_2(x)= \frac{\sin(4\pi/3)}{2}\sum_{k=0}^{+\infty}\frac{(-1)^k(2\pi  \left( x- \frac{2}{3} \right))^{2k+1}}{(2k+1)!} \frac{2^{2k+1}}{2^{2k+1}+1} -$$
\begin{equation} \label{khj} \frac{\cos(4\pi/3)}{2}\sum_{k=1}^{+\infty}\frac{(-1)^k(2\pi\left( x- \frac{2}{3} \right))^{2k}}{(2k)!}\frac{2^{2k}}{2^{2k}-1}.\end{equation}

As $V_1(x)=V_2(1-x)$ a similar result can be derived for $V_1$ (which can be expressed in power series around $1/3$). We plot in Figure \ref{fig:my_label2}  the expression of the subaction via the  $1/2$ iterative procedure and via the power expansion described above.
In Figure \ref{fig:RR7} we plot the graph of $R$ we get via the  $1/2$ iterative procedure.


\section{The case \(A(x)=\sin(2 \pi x)\)} \label{kk}
Now we consider the potential \(A(x)=\sin(2 \pi x)\) and that $T(x)= 2\,x$ (mod 1) acts on $[0,1]$. Consider also the inverse branches of $T$ given by $\tau_1(x)=\frac{x}{2}$ and $\tau_2(x)=\frac{x+1}{2}$. In page 23 in \cite{Conze} the authors conjectured that
in this case the maximizing probability has support on the periodic orbit of period $4$ given by
$\{1/15, 2/15, 4/15, 8/15\}.$ The graph for the subaction $V$ we obtain from the  $1/2$ iterative procedure for such $A$ is presented in Figure \ref{fig:my_label4}.  Although at first glance there seem to be $5$ functions $V_j$ in $[0,1]$ we point out that from the point of view of $S^1$ (periodic) there exists just $4$. The left one is just a continuation of the most right one. This is consistent with the supposition that the maximizing probability has support on a periodic orbit of period $4$. Note that  in the present case we do not  know  the value $m(A).$  In \cite{Conze} the authors conjectured  that $m(A)= \frac{ A(1/15) + A(2/15) +  A(4/15) + A(8/15) }{4}.$  It is possible to show that the conjecture is true.
\begin{figure}[h]
    \centering
    \includegraphics[height=3cm,scale=0.35]{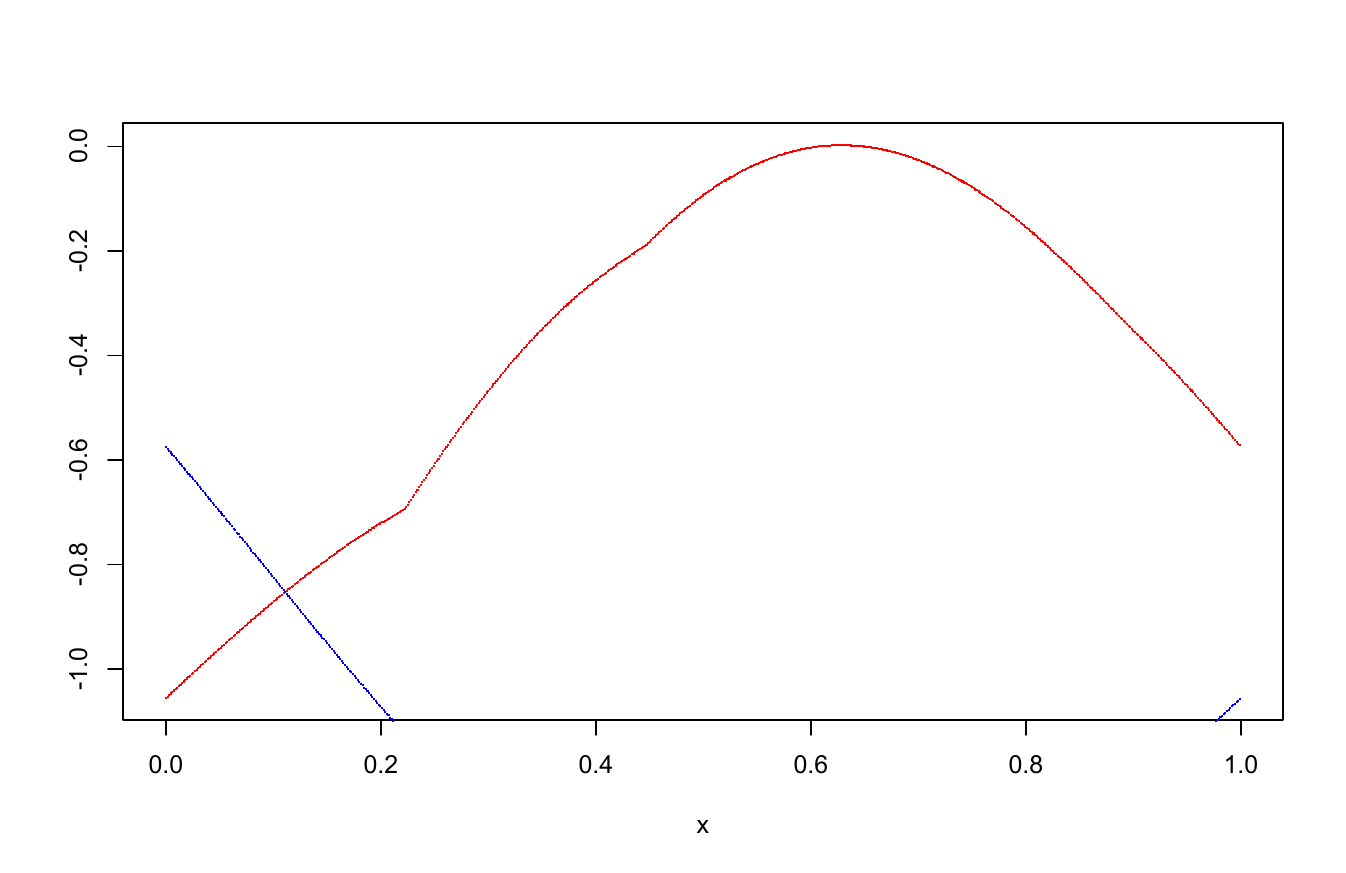}
    \caption{Case $\sin(2 \pi x)$ - In blue we show the graph of the subaction we get from the $ 1/2$ iterative procedure when the calibrated subaction equation is realized by the branch $\tau_2$. In red when it is realized by  the branch $\tau_1$. The graph of the approximation of the calibrated subaction  $V$ is the supremum of the blue and red graphs.}
\label{fig:my_label4}
\end{figure}
In order to do the computations, we consider the $[0,1]$ point of view.
From the graph we obtained via the computer it is natural to try to obtain $V$ via the expression
$V(x)= \sup\{V_1(x), V_2(x), V_3(x),V_4(x), V_5(x)\}$. Examining the Figure~\ref{fig:my_label4} we realize the following relations
$$V_5(x) +\hat{m} (A) = V_4(\tau_1(x))+A(\tau_1(x)),V_4(x)+ \hat{m}(A)=V_3(\tau_1(x))+A(\tau_1(x)),\,\,$$
$$V_3(x)+\hat{m}(A)=V_2(\tau_1(x))+A(\tau_1(x)),\,\,V_2(x)+\hat{m}(A)=V_1(\tau_1(x)+A(\tau_1(x)),$$
$$\,\,\,\text{and}\,\,\,V_1(x)+\hat{m}(A)=V_4(\tau_2(x))+A(\tau_2(x)).$$

The analysis of this case is similar to the previous one. We will just outline
the proof. In order to simplify the analytic expressions on this section (that depends on adding constants) we will write an expression like
$ V_j(x)= \lim_{n \rightarrow \infty}\sum_{i=0}^{n-1}  (F\circ \eta^i(x) - \,K) = \sum_{i=0}^{\infty}  (F\circ \eta^i(x) - \,K),$
$j=1,2,...,r$, on the form
\begin{equation} \label{falt}
 V_j(x)\cong  \sum_{i=0}^{\infty}  F\circ \eta^i(x) .
\end{equation}

\begin{figure}[h]
    \centering
    \includegraphics[height=3cm,scale=0.35] {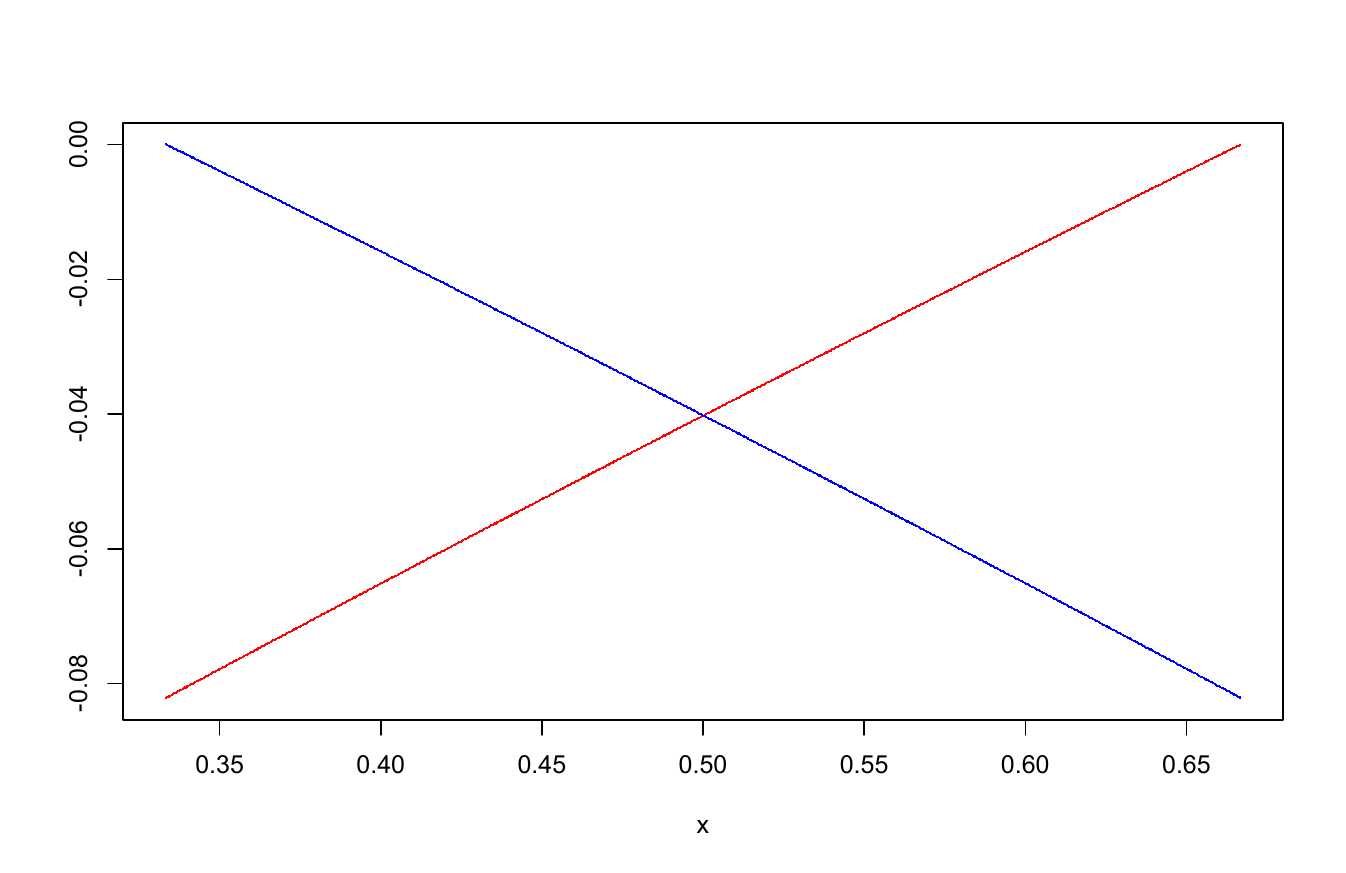}
    \caption{Graphs for the first example of approximating the spectral  spectral radius. The graph in red describes the values where the  $1/2$ iterative procedure detects that the maximization (in right side of the subaction equation) was obtained via the branch $\tau_1$. The function with this graph is denoted by  $V_1$.  The graph in  blue describes the cases where the mazimization was obtained via the branch $\tau_1$. The function $V$ which is the maximum of $V_1$ and $V_2$  is the calibrated subaction one gets from the  $1/2$ iterative procedure. There is a perfect match of such $V$ with the expression (\ref{expr_ex5}).}
    \label{fig:matriz_algoritmo_sob}
\end{figure}
If one gets the explicit expression for $V_1$ it will follow from the system above that we can also get the  explicit expressions for
$V_2,V_3,V_4$. We will show later that
\begin{equation} \label{sisi}
V_1(x)=\sum_{m=0}^{+\infty} \sum_{j=0}^{3}\left[ \sin \left(\frac{\pi}{2^{(j+4m)}} \left(\frac{2^{4(m+1)}-1}{2^4-1}+x \right )\right) - \sin\left(2 \pi \frac{2^m}{15}\right) \right].
\end{equation}
Assuming that  the above relations among the $V_j$ are true we get
\begin{equation*}V_1(x)-V\circ \tau_1^3 \circ \tau_2 (x) =A\circ \tau_1^3 \circ \tau_2 (x)+ A\circ \tau_1^2 \circ \tau_2 (x) + A \circ \tau_1 \circ \tau_2 (x)+A\circ \tau_2(x) - 4 \hat{m}(A).\end{equation*}
Now, we take \(\eta(x)=\tau_1^3\circ \tau_2(x)\), and $F(x)=A\circ \tau_1^3 \circ \tau_2 (x)+ A\circ \tau_1^2 \circ \tau_2 (x) + A \circ \tau_1 \circ \tau_2 (x)+A\circ \tau_2(x),$ with \(K=4\, \hat{m}(A)\). Then, we get $\eta(x)=\frac{x}{2^4}+\frac{1}{2^4}$.

Note that if $x \in [0,1]$, then
$\displaystyle\lim_{n \rightarrow + \infty} \eta^n(x)=\frac{1}{15}$. In this way we get numerical evidence that
$\displaystyle\hat{m}(A)=\lim_{n\rightarrow +\infty}\frac{F(\eta^n(x))}{4}=
\frac{F(1/15)}{4}\approx 0.4841$. This is consistent with the value
$m(A)=\frac{A(1/15)+A(2/15)+A(4/15)+A(8/15)}{4}\approx 0.4841$. Using the truncated expression we get
$V_1^{n*}(x)=\sum_{i=0}^{n-1}F(\eta^i(x)),\,\,\,\text{and}$
$V_2^{n*}(x)=V_1(\tau_1(x))+A(\tau_1(x)) -m(A)$. Applying the above reasoning in a recursive way we obtain an expression for  $V_1(x)\cong$
\begin{equation}
   \sum_{i=0}^{+\infty}A(\frac{\eta^i(x)+1}{2^4})+A(\frac{\eta^i(x)+1}{2^3})+A(\frac{\eta^i(x)+1}{2^2})+A(\frac{\eta^i(x)+1}{2}).
\end{equation}
Therefore,
\begin{equation}  \label{BB1}
    V_1(x)\cong\sum_{i=0}^{+\infty}\sin(\pi ((x+1)/2^{i})).
\end{equation}
The function $V_2$ can be obtained from $V_1$. The function $V_3$ from $V_2$ and so on.
One can show that
$V(x)= \sup\{V_1(x), V_2(x), V_3(x),V_4(x), V_5(x)\}$
is a calibrated subaction for $A$ and that $\tilde{m}(A)= m(A).$

\section{Estimation of the joint spectral radius: two examples and a more general analytic expression} \label{jsp}

In the class of  examples we consider here does not exists  a map acting on $[0, 1]$ but it is naturally  defined two inverse branches (an iterated function system). Anyway, the $1/2$ iterative procedure will produce useful information.

Consider
\[A_1=\begin{pmatrix}
a_1 & b_1\\
c_1 &d_1
\end{pmatrix} \quad \quad A_2=\begin{pmatrix}
a_2 & b_2\\
c_2&d_2
\end{pmatrix},\]
with
\[\tau_1(x)=\frac{(a_1-b_1)x+b_1}{(a_1+c_1-d_1-b_1)x +b_1+d_1}\]
and
\[\tau_2(x)=\frac{(a_2-b_2)x+b_2}{(a_2+c_2-d_2-b_2)x +b_2+d_2}.\]

Take $I_1=\tau_1([0,1])$, $I_2=\tau_2([0,1])$ and define the potential

\[A(x)=\left\{\begin{matrix}
1/2 \,\,(\log|(\tau_1^{-1})'(x)|+\log(\det(A_1))\,),\quad x\in I_1,\\
1/2 \,\,(\,\log|(\tau_1^{-1})'(x)|+\log(\det(A_2))\,),\quad x\in I_2.
\end{matrix}\right.\]
In  \cite{JP} the authors explain how the joint spectral radius can be analyzed from the point of view of Ergodic Optimization.
The special space of  ``invariant probabilities'' to be considered on this case is described on Definition 7 of  \cite{JP}. It follows from results on  \cite{JP} (see expression (42)) that the value $e^{m(A)}$ (\,$m(A)$ is obtained in a similar way as in classical Ergodic Optimization\,) is equal to the joint spectral radius $\rho(A_1,A_2)$ (under some conditions for $A_1,A_2$). In this section the main issue is to estimate $m(A)$. In the first and second examples below the subaction is rigorously obtained. We will estimate in our first example the value $m(A)$ using the  $1/2$ iterative procedure.
\begin{figure}[h]
    \centering
    \includegraphics[height=3cm,scale=0.35]{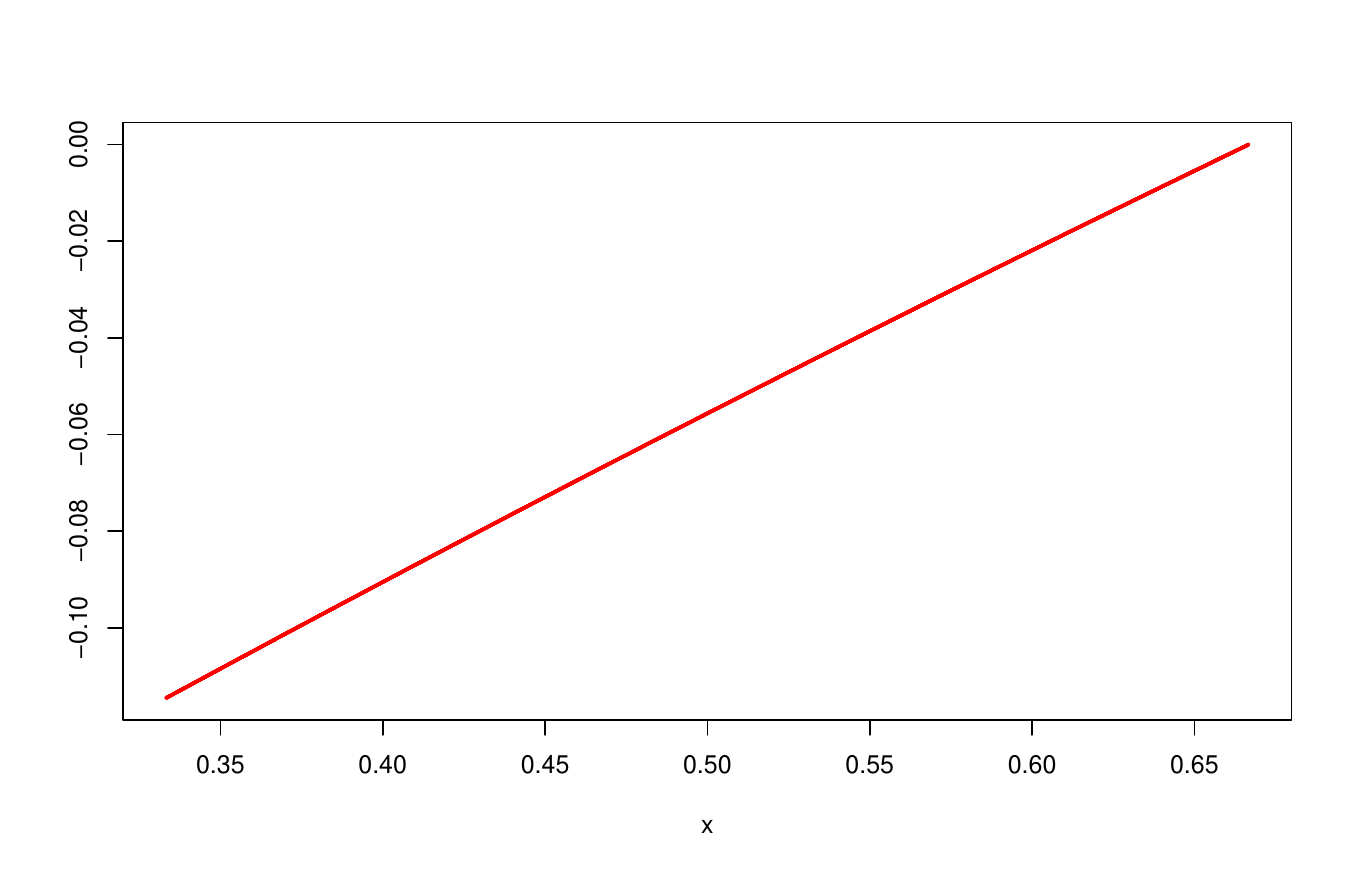}
    \caption{Second example in the case of the estimation of the joint spectral radius. In this case from the  $1/2$ iterative procedure we get a picture that indicates that  the realizer is always $\tau_1$ (red color).}
    \label{fig:matriz_exemplo6}
\end{figure}

We consider the first example. Take
\[A_1=\begin{pmatrix}
2 & 1\\
2 &2
\end{pmatrix} \quad \text{ and }\quad A_2=\begin{pmatrix}
2 & 2\\
1&2
\end{pmatrix}.\]

In this case the inverse branches are $\tau_1(x)=\frac{x+1}{x+3}$ e $\tau_2(x)=\frac{2}{4-x}$.

The potential is given by

\[A(x)=\left\{\begin{matrix}
\,\,1/2 \,(\,\log(|\frac{2}{(x-1)^2}|)+\log(2)\,),\quad 1/3\leq x\leq 1/2,\\
1/2 \,\, (\log(|\frac{2}{x^2}|)+\log(2)\,),\quad 1/2\leq x\leq 2/3.\\
\end{matrix}\right.\]
Applying a high order iteration of  $1/2$ iterative procedure $\mathcal{G}^n (f_0)$ we get an output called ``subaction'' which helps to find the value $m(A).$ This is in agreement  to what was predicted by the theory in \cite{JP}. Corollaries 13 and 14 of  \cite{JP} describe  the values of the joint spectral radius $\rho(A_1,t\,A_2)$, for some values of $t>0$. The value  $m(A)\approx 1.2702$  does not correspond to the spectral radius of  either $A_1$ or $A_2$  (they are equal). This is in agreement  to what was predicted by the theory in \cite{JP}. Looking Figure \ref{fig:matriz_algoritmo_sob} which was  obtained  from the  $1/2$ iterative procedure (showing  the possible realizers) we assume that we should take $V_1,V_2$ (with realizers, respectively,
$\tau_1$ and $\tau_2$) satisfying

$V_2(x)+\hat{m}(A)=V_1(\tau_1(x))+A(\tau_1(x)) ,\,\,V_1(x)+\hat{m}(A)=V_2(\tau_2(x))+A(\tau_2(x)).$

Finally, we get
 \begin{equation} \label{eq:log1}\, V_2(x)-V_2\circ \tau_2 \circ \tau_1(x)= A\circ \tau_2 \circ \tau_1(x)+A\circ \tau_1(x) -2\hat{m}(A).
 \end{equation}
Figure~\ref{fig:matriz_algoritmo_sob} shows the pictures of the graphs of the functions $V_1$ and $V_2$. As $q=\frac{1}{2}(\sqrt{17}-3)$ is the fixed point of  $\tau_2 \circ \tau_1$ we obtain
$\hat{m}(A)=\frac{A\circ \tau_2 \circ \tau_1(q)+A\circ \tau_1(q)}{2}$
$= \frac{ A(\frac{1}{2}(\sqrt{17}-3))+A(\frac{1}{2}(5- \sqrt{17}))}{2}$
$=\frac{1}{4}\left( 2\log(2)+\log(2/(q-1)^2)+\log(2/(q^2)) \right)\approx 1.2702$.

\begin{figure}[h]
    \centering
    \includegraphics[height=5cm,scale=0.45]{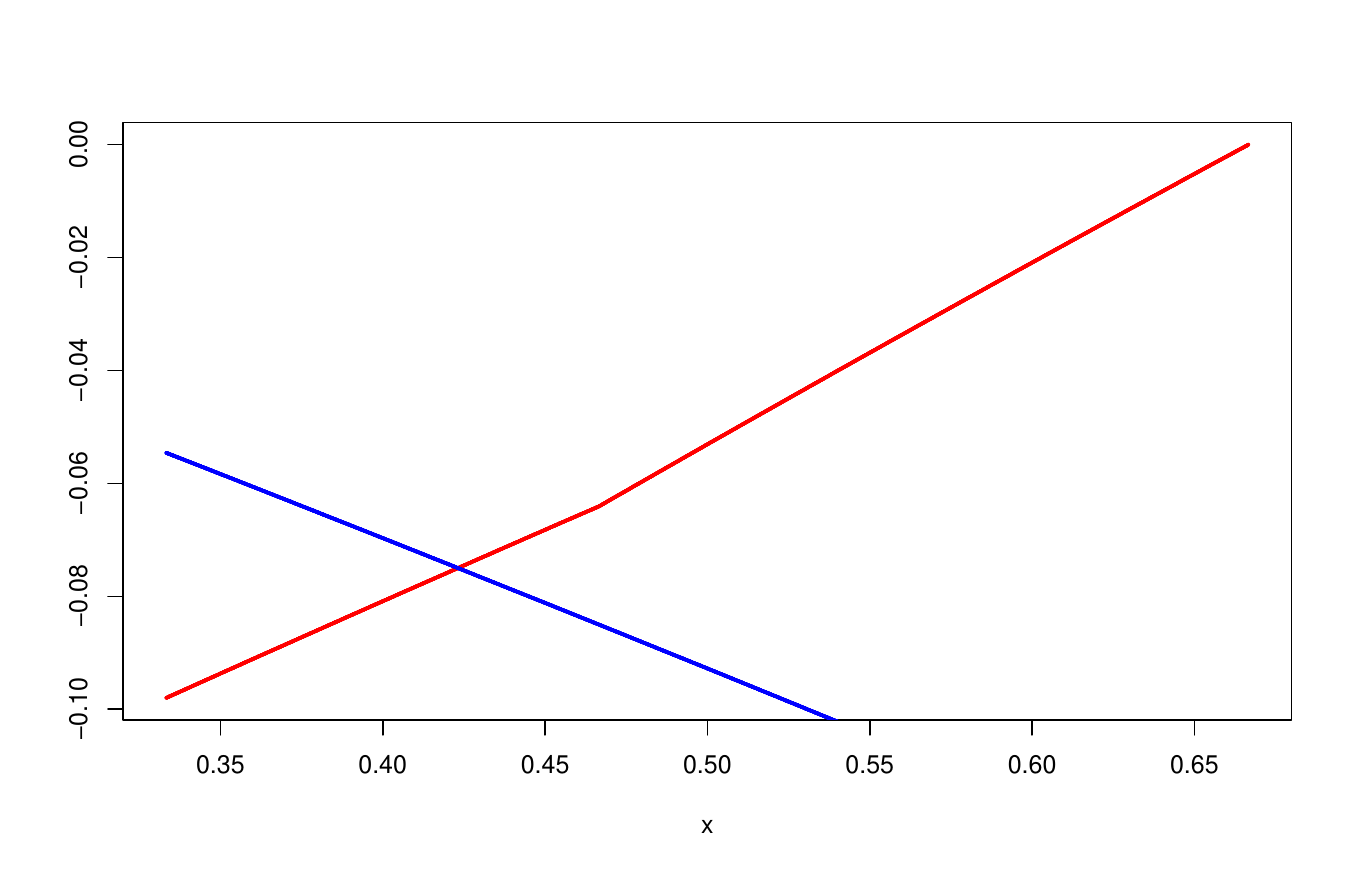}
    \caption{Joint spectral radius case -The graph in blue indicates where the realizer  is attained by $\tau_2$ and in red by  $\tau_1$. Here   the value of the parameter is  equal to $t=0.92$.}
    \label{fig:matriz_exemplo7_t2}
\end{figure}

This is in agreement with the value we get from the  $1/2$ iterative procedure. Therefore,  the  $1/2$ iterative procedure is able to estimate  the joint spectral radius $\rho(A_1,A_2).$ After some computations we will show later that $m(A)$ satisfies $m(A)=\log\left( \frac{1}{2}(3+\sqrt{17})\right)$, and taking $b=\frac{1}{2} (3 + \sqrt{17})$
we will finally get that
$V(x)=\max \{\log(x+b),\log(1-x+b) \}$
is a subaction. Now we will begin the computations for this case. Taking  $F(x)=A\circ \tau_2 \circ \tau_1(x)+A\circ \tau_1(x)$ and  $\eta(x)=\tau_2 \circ \tau_1(x)$, we get $\displaystyle V_2(x)\cong \lim_{n \rightarrow +\infty} \sum_{i=0}^{n} F \circ \eta^i(x).$ This means
\begin{equation} \label{quec} V_2(x)\cong2\log\left(\prod_{i=0}^{+\infty}(11+3 \eta^i(x))\right).
\end{equation}
We note that from equation  (\ref{eq:log1}) we get
$V_1(x)=V_2(1-x).$ One can also show  that in this case  the piecewise analytic expression for the calibrated subaction $V$ can given by $V(x)=\max $ of
\begin{equation}\label{expr_ex5}
\left\{ \log\prod_{i=0}^{\infty}\left(\frac{11+3 (\tau_2\circ \tau_1)^i(x)}{11+(\frac{3}{2}\left(\sqrt{17}-3\right))} \right),\log\prod_{i=0}^{\infty} \left(\frac{11+3 (\tau_2\circ \tau_1)^i(1-x)}{11+(\frac{3}{2}\left(\sqrt{17}-3\right))} \right)  \right\}
\end{equation}
There is a quite strong simplification of all this. Indeed, we get that in this case the subaction $V$ satisfies
$V(x)=\max\{V_1(x),V_2(x)\},$
where $V_2(x)=\log\left(h(x)\right)$ for some function $h$. From the information we get from the  $1/2$ iterative procedure it seems that $h$ is linear. Assuming that  $V_2(x)=\log(x+b)$ we get the system
$\log\left(\frac{(b+x)(11+3x)}{b(11+3x)+6+2x}\right)= \log \left( (11+3x)e^{-2m(A)} \right)$.
This means $e^{-2m(A)}=\frac{b+x}{6+11b+2x+3bx}.$
As $m(A)$ satisfies $m(A)=\log\left( \frac{1}{2}(3+\sqrt{17})\right)$,
taking derivative on $x$ and using the condition to be equal to  zero we get $6+11b+2x+3bx-(2+3b)(b+x)=0,$ that is $6+9b-3b^2=0.$
Finally, we get $b=\frac{1}{2} (3 + \sqrt{17}).$
Note that $b=e^{m(A)}$, therefore we get the candidate for subaction
$V(x)=\max \{\log(x+b),\log(1-x+b) \}=  \max\{V_2(x),V_1(x)\}.$
In order to check that this $V$ is indeed the solution we plug the above expression for $V_2$  in equation \eqref{eq:log1} and we have a confirmation that such $V$ is  a subaction.
\smallskip

We will consider now our second  example. Denote
\[A_1=\begin{pmatrix}
2 & 1\\
 2& 2
\end{pmatrix} \quad \text{and}  \quad A_2= \begin{pmatrix}
1& 1\\
 1/2& 1
\end{pmatrix}.\]

In this case $\tau_1(x)=\frac{x+1}{x+3}$ and $\tau_2(x)=\frac{2}{4-x}$, and

\[A(x)=\left\{\begin{matrix}
(1/2)(\log(|\frac{2}{(x-1)^2}|)+\log(2)),\quad 1/3\leq x\leq 1/2\\
(1/2)(\log(|\frac{2}{x^2}|)-\log(2)),\quad 1/2\leq x\leq 2/3\\
\end{matrix}\right.\]
From Corollaries 13 and 14 of  \cite{JP} it follows that the joint spectral radius $\rho(A_1,A_2)$ is equal to the spectral radius $\rho(A_1)$ of $A_1$ which is $ 2 + \sqrt{2}$. This corresponds to $m(A)= \log( 2 + \sqrt{2})$. Via the  $1/2$ iterative procedure we obtain the value $m(A )\approx 1.2279 \sim \log ( 2 + \sqrt{2})$ after 14 iterations of $\mathcal{G}$ applied to $f_0=0$.   This is in agreement with the analytical result presented in \cite{JP}. Looking Figure \ref{fig:matriz_exemplo6} (which was  obtained  from the  $1/2$ iterative procedure)
and proceeding in the same way as before we get
$V(x)-V\circ \tau_1(x)=A\circ \tau_1 (x) -m(A).$
As $q= \sqrt{2}-1 $ is a fixed point of  $\tau_1$ we finally get
$m(A)=A(\tau_1(q))= \log\left( \frac{2}{2- \sqrt{2}} \right)=\log(2+\sqrt{2}) \approx 1.22795.$
Proceeding in the same way as in  the previous example we get

$\,\,\,\,\,\,V(x) \cong\sum_{i=1}^{+\infty} -\log \left( 1-\tau_1^i(x) \right)= (-1)\log \prod_{i=1}^{+ \infty} \left( 1- \tau_1^i(x) \right).$

In the cases where we get explicit estimations, the approximation of the exact value $m(A)$ to  four decimals places, required $30$ iterations. With $15$ iterations we get an approximation to  two decimal places.

Now, we consider a more general case. Given $t>0$, denote
\[A_1=  \begin{pmatrix}
2 & 1\\
2 &2
\end{pmatrix} \quad \text{ and }\quad tA_2=\begin{pmatrix}
2\,t & 2\,t\\
1\, t&2\, t
\end{pmatrix}.\]

In this case $\tau_1(x)=\frac{x+1}{x+3}$ and $\tau_2(x)=\frac{2}{4-x}$.
As  $t>0$, then
\[A(x,t)=\left\{\begin{matrix}
(1/2)(\log(|\frac{2}{(x-1)^2}|)+\log(2)),\quad 1/3\leq x\leq 1/2,\\
(1/2)(\log(|\frac{2}{x^2}|)+\log(2t^2)),\quad 1/2\leq x\leq 2/3.\\
\end{matrix}\right.\]
We know from the other cases we already consider that for the values $t=1/2$ and $t=1$ we get different maximal values $ m(A)$ and different subactions. Denote by $m(A,t)$ the function which gives the maximal value  of $A(x,t)$ (where $e^{m(A,t)} $ is the joint spectral radius $\rho(A_1,\,t\, A_2)$\,), for each $ t>0$. We are not able to obtain in a rigorous manner the subaction for all cases of $t>0$. However, we are able to show rigorously  that there is an interval $ 0\leq t\leq  \frac{4(4+3\sqrt{2})}{18+13\sqrt{2}} $ where the maximal value is constant  (see computations of case 1 below). Via the  $1/2$ iterative procedure we will be able to plot (a non rigorous estimation)  the maximal value as a function of $t$ (see figures \ref{fig:matriz_exemplo7_comparacao} and  \ref{fig:mylabel3}).
\begin{figure}[h]
    \centering
    \includegraphics[height=3cm,scale=0.35]{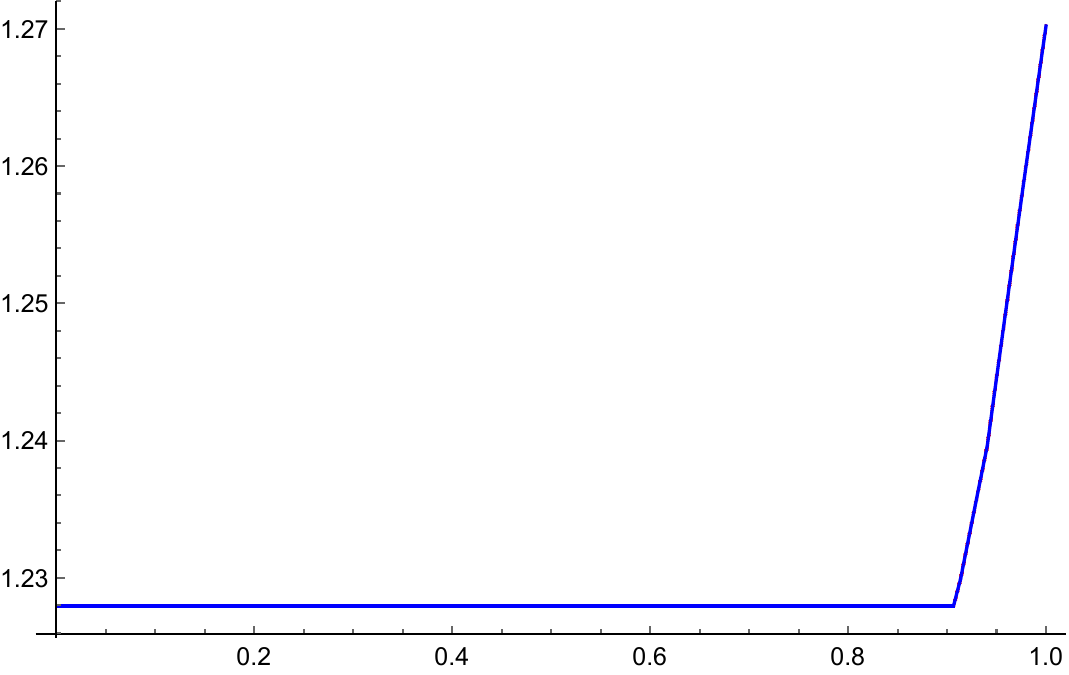}
    \caption{Joint spectral radius case - In blue we plot the graph of $m(A,t)$ as a function of $t$ via  expressions we get explicitly here and in blue the estimation of $m(A,t)$ via the  $1/2$ iterative procedure. There is a perfect match in some intervals in $[0,1]$.}
    \label{fig:matriz_exemplo7_comparacao}
\end{figure}
The main idea here is to try to take one of the $V_i$ in the form $V_i(x) = \log(x+b)$
(or, $\log (b-x)$). For guessing the number of $V_i$, $i=1,2...,r$, we use (in most of the cases) the picture we get from the  $1/2$ iterative procedure.\\
{\bf Case 1:} It will follow that $ 0\leq t\leq  \frac{4(4+3\sqrt{2})}{18+13\sqrt{2}}$.
We will  get here explicitly  that $m(A,t)= \log(2+\sqrt{2})$, when
$ 0\leq t\leq  \frac{4(4+3\sqrt{2})}{18+13\sqrt{2}} $. We will elaborate on that. For small values  $t\sim 0$, the value $m(A,t)$ we get on the computer indicates that $m(A,t)=\log(2+\sqrt{2})$. Moreover, it suggests that in order to get the calibrated subaction $V$ we should  work with two $V_i$:
\begin{equation} \label{logeqt1}V_1(x,t)+\hat{m}(A,t)=A(\tau_2(x),t)+V_2(\tau_2(x),t), \end{equation}
\begin{equation}\label{logeqt2} V_2(x,t)+\hat{m}(A,t)=A(\tau_1(x),t)+V_2(\tau_1(x),t). \end{equation}
$V(x,t)=\max\{V_1(x,t),V_2(x,t) \}$ is the candidate to be the subaction for  $A(x,t)$.  As $m(A,t)$ seems to be constant in an interval  and $A(\tau_1(x),t)=\log(\frac{2}{1-\tau_1(x)})$ we conclude that $V_2$ should not depend on $t$. We assume  $V_2(x,t)=\log(x+b)$ and  then from last equation we get $b= 1+\sqrt{2}$ and finally
$V_2(x,t)=\log(x+1+\sqrt{2}).$ It is easy to confirm that $V_2(x,t)+\log(2+\sqrt{2})=V_2(\tau_1(x),t)+A(\tau_1(x),t)$. Making a substitution in (\ref{logeqt1}) we get $V_1(x,t)= \log\left(t(2+\sqrt{2}-\frac{x}{\sqrt{2}})\right)  $. Clearly $\hat{m}(A,t)=\log(2+\sqrt{2})$ is a natural  candidate to be $m(A,t)$.
We may ask which values of $t$ the above expressions for $V_1$ and $V_2$ are such that the subaction $V$ is given by
\begin{equation} \label{iur} V(x,t)=\max[V_1(x,t),V_2(x,t)]?
\end{equation}
In particular we get
$A(\tau_1(x),t)+V_2(\tau_1(x),t)=V_2(x,t)+\hat{m}(A,t),\,\,$ and
$A(\tau_2(x),t)+V_2(\tau_2(x))=V_1(x,t)+\hat{m}(A,t).$

Given $x \in [1/3,2/3]$ and $i \in \{1,2 \}$, then for some $j \in \{1,2 \}$, we get
\[A(\tau_i(x),t)+V_1(\tau_i(x))\leq V_j(x,t)+\hat{m}(A,t). \]
That is, $\max \left \{\log \left( t (3+x) \left(2 +
\frac{1}{\sqrt{2}}+ \frac{\sqrt{2}}{3+x} \right) \right), \log \left( t^2 (4-x) \left( 2 +\sqrt{2}-\frac{\sqrt{2}}{4-x} \right) \right)  \right \} \leq $

$\leq \max \left \{\log\left(t(2+\sqrt{2}-
\frac{x}{\sqrt{2}})(2+\sqrt{2})\right),\log\left((1+\sqrt{2}+x))(2+\sqrt{2})\right) \right \}.$

On the other hand if $x \in [1/3, \frac{\sqrt{2}}{2+\sqrt{2}} ]$ then

$\,\,\,\,\,\,\,\,\, t(3+x) \left(2 + \frac{1}{\sqrt{2}}+ \frac{\sqrt{2}}{3+x} \right) \leq  \left(t(2+\sqrt{2}-\frac{x}{\sqrt{2}})(2+\sqrt{2})\right).$

Therefore, in this interval  \(A(\tau_1(x),t) +V_1(\tau_1(x),t ) \leq V(x,t)+\hat{m}(A,t) \).
Now, consider $x \in [x(t),2/3]$, where $x(t)$ is the point such that

$\,\,\,\,\,\,\,\,\,t (3+x(t)) \left(2 + \frac{1}{\sqrt{2}}+ \frac{\sqrt{2}}{3+x(t)} \right)=\left((1+\sqrt{2}+x(t)))(2+\sqrt{2})\right) .$

This means that if $x \in [x(t),2/3]$, then,  \(A(\tau_1(x),t) +V_1(\tau_1(x),t ) \leq V(x,t)+\hat{m}(A,t)\). From this follows that $x(t)\leq \frac{\sqrt{2}}{2+\sqrt{2}}= 0.414214... $, Then,  for  $x \in [1/3,2/3]$ we get
$A(\tau_1(x),t) +V_1(\tau_1(x),t ) \leq V(x,t)+\hat{m}(A,t)$. This condition is satisfied for $t \leq \frac{4(4+3\sqrt{2})}{18+13\sqrt{2}} \approx 0.9061$. It is compatible with the information we get from the  $1/2$ iterative procedure. Now we will show that  \(A(\tau_2(x),t) +V_1(\tau_2(x),t ) \leq V(x,t)+\hat{m}(A,t) \) for such values of $t$. Note that if $ 0\leq t \leq \frac{(2+\sqrt{2})^2}{8+3\sqrt{2}}\approx 0.952 $, then,
$t^2 (4-x) \left( 2 +\sqrt{2}-\frac{\sqrt{2}}{4-x} \right) \leq t(2+\sqrt{2}-\frac{x}{\sqrt{2}})(2+\sqrt{2}).$
Therefore, $V(x,t)$ given by equation (\ref{iur}) is a calibrated subaction with  $m(A,t)= \log(2+\sqrt{2})$, as far as, $ 0\leq t\leq  \frac{4(4+3\sqrt{2})}{18+13\sqrt{2}} $.
The final conclusions is that $m(A,t)=\log(2+\sqrt{2})$ for  $t \in [0,t_1]$, where  $ t_1:= \frac{4(4+3\sqrt{2})}{18+13\sqrt{2}} \approx 0.9061 $. In Figure~\ref{fig:mylabel3} we show a detailed estimation of the graph of $m(A,t)$ (via the  $1/2$ iterative procedure)  for $t$ close to $t_1$.
\begin{figure}[h]
    \centering
    \includegraphics[height=4cm,scale=0.6]{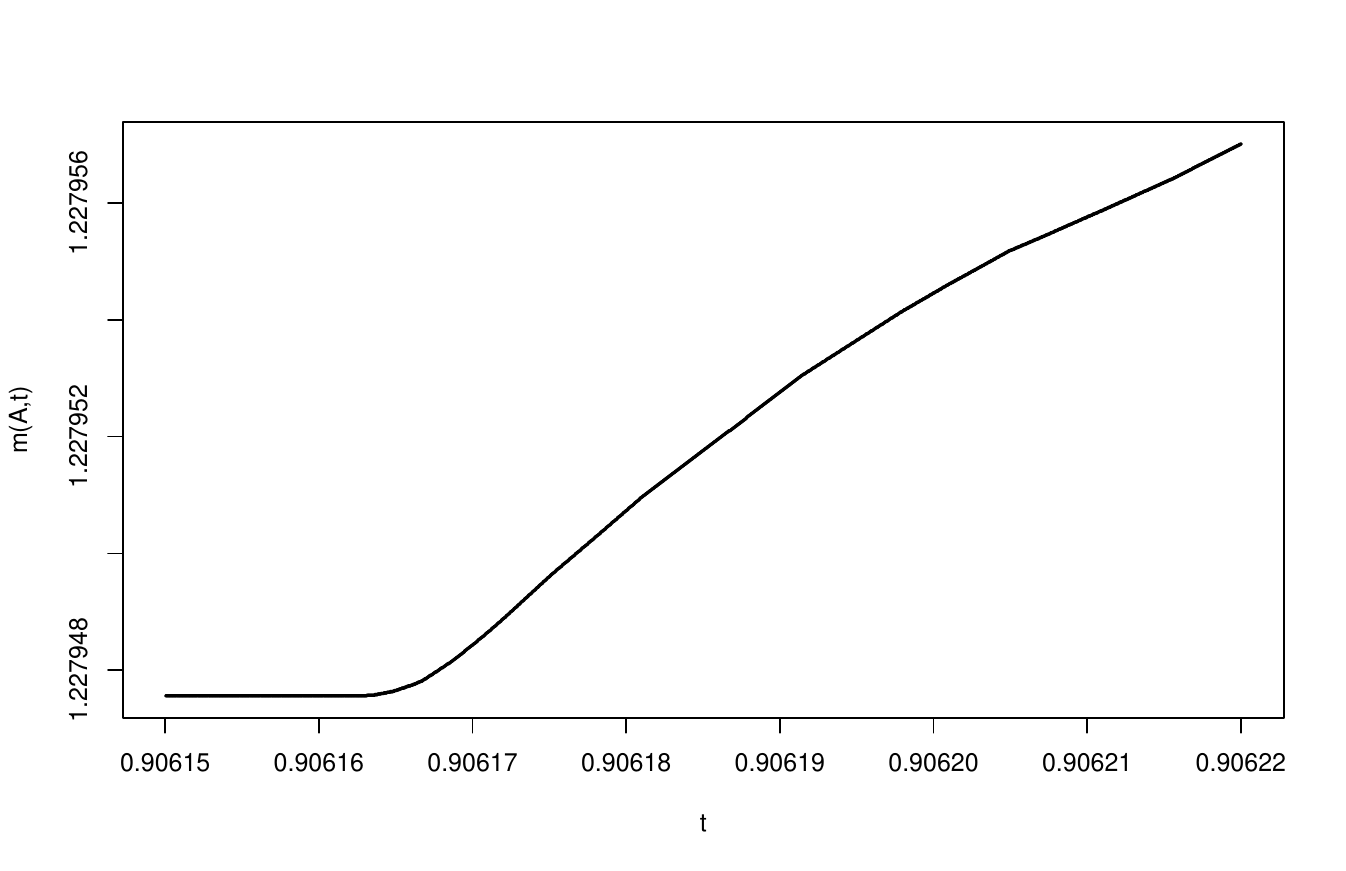}
    \caption{Graph of $ m(A,t)$ for $t$  around the point $t_1$ where  $m(A,t)$ is not constant anymore.}
    \label{fig:mylabel3}
\end{figure}

{\bf Case 2:} Now we analyze parameters close to $ t=0.91$. In this case the picture we get  was not good enough for a guess. But, the  approximated value $\hat{m}(A)\approx 1.228902$ suggests an orbit of period  $4$ as the maximizing probability. In this way is natural to try to obtain $V$ using the system:
$V_4(x,t)+\hat{m}(A,t)=V_3(\tau_1(x),t) + A (\tau_1(x),t),\,\,
V_3(x,t)+\hat{m}(A,t)=V_2(\tau_1(x),t) + A (\tau_1(x),t),$
$V_2(x,t)+\hat{m}(A,t)=V_1(\tau_1(x),t) + A (\tau_1(x),t),
V_1(x,t)+\hat{m}(A,t)=V_4(\tau_2(x),t) + A (\tau_2(x),t).$
We will try to get $V$ via
\begin{equation} \label{eli}
V(x,t)=\max[V_1(x,t),V_2(x,t),V_3(x,t),V_4(x,t)].
\end{equation}
This system (if accomplish its mission of getting $V$) gives the exact value
\begin{equation} \label{elis}
m(A,t)=\frac{1}{4}\log\left((75+\sqrt{5609})t\right).\end{equation}

This value we get  from the  $1/2$ iterative procedure -  for the estimation of $m(A,t)$ - when $t=0.91$ is quite close to the one we get from the above analytical expression for this value.

After a tedious computation we get:

$V_1(x,t)=\log\left(b-x\right)$,
$V_2(x,t)= \log (d(t)(-1 - x + b (3 + x))$,
$V_3(x,t)= $ $\log \left(2 d(t)^2 (-2 - x + b (5 + 2 x)) \right)$,
$V_4(x,t)= \log (2 d(t)^3 (-7 - 3 x + b (17 + 7 x))),$
where, $b=\frac{1}{34}\left( 89 +\sqrt{5609} \right)$. We checked that  $V(x,t)$ is indeed the subaction when  $t \in [t_2, t_3]$, where approximately $[t_2,t_3]= [0.908571, \,0.912996]$.
More precisely, one can get
\[t_2:=\frac{367765714335 - 4904055941 \sqrt{5609}}{533794816},\]
\[\,\,\,\,\, \text{ and } t_3:=\frac{1900479599391 + 25366638853\sqrt{5609}}{4162416040000}\]
The value $m(A,t)$ is given by  (\ref{elis}).
\begin{figure}[h]
    \centering
    \includegraphics[height=3cm,scale=0.55]{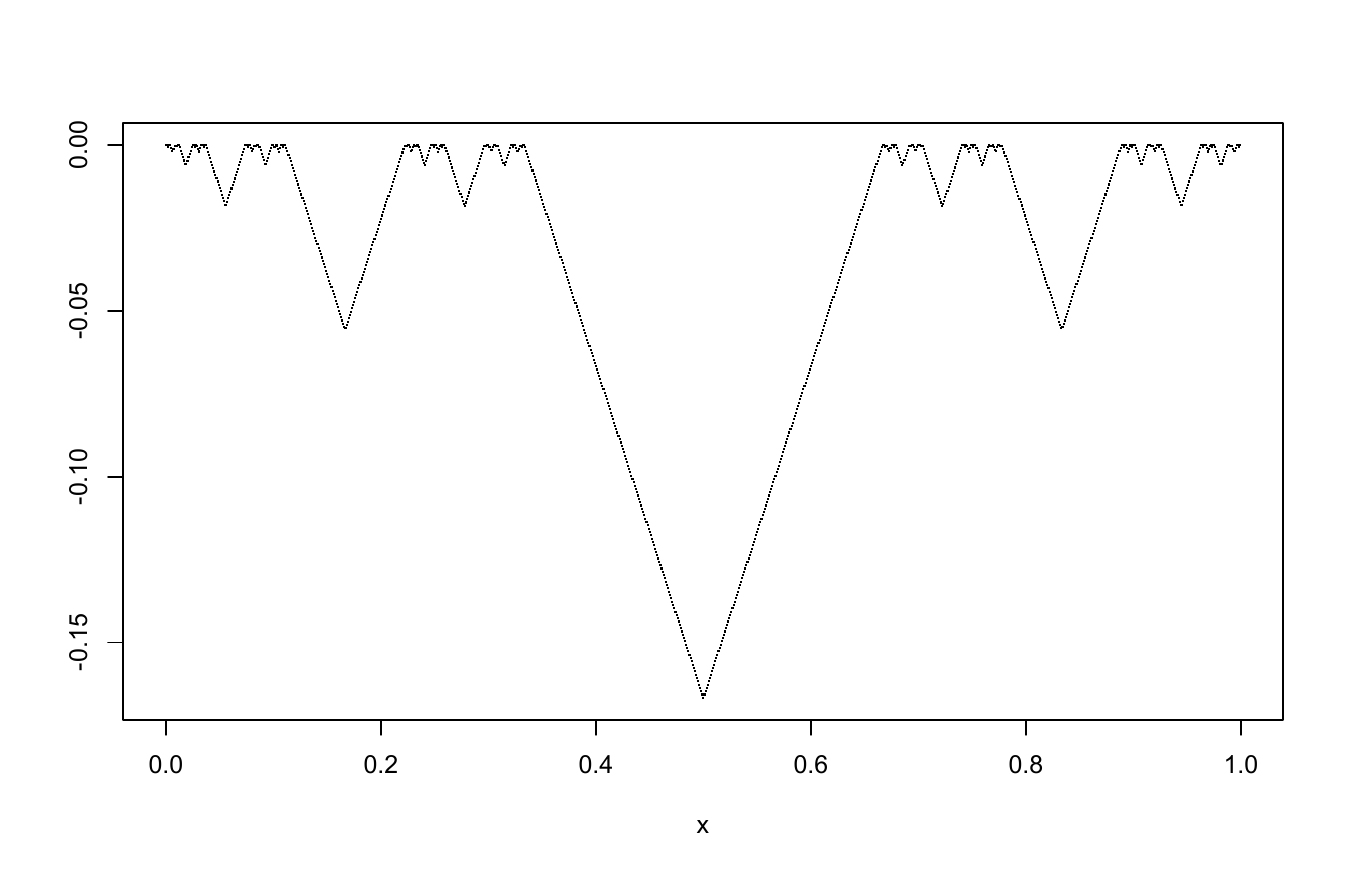}
    \caption{\textit{Graph of the truncation $A_{100}(x)$ in a discretization of $10^5$ points}}
    \label{fig:my_label4}
\end{figure}
\begin{figure}[h]
    \centering
    \includegraphics[height=3cm,scale=0.55]{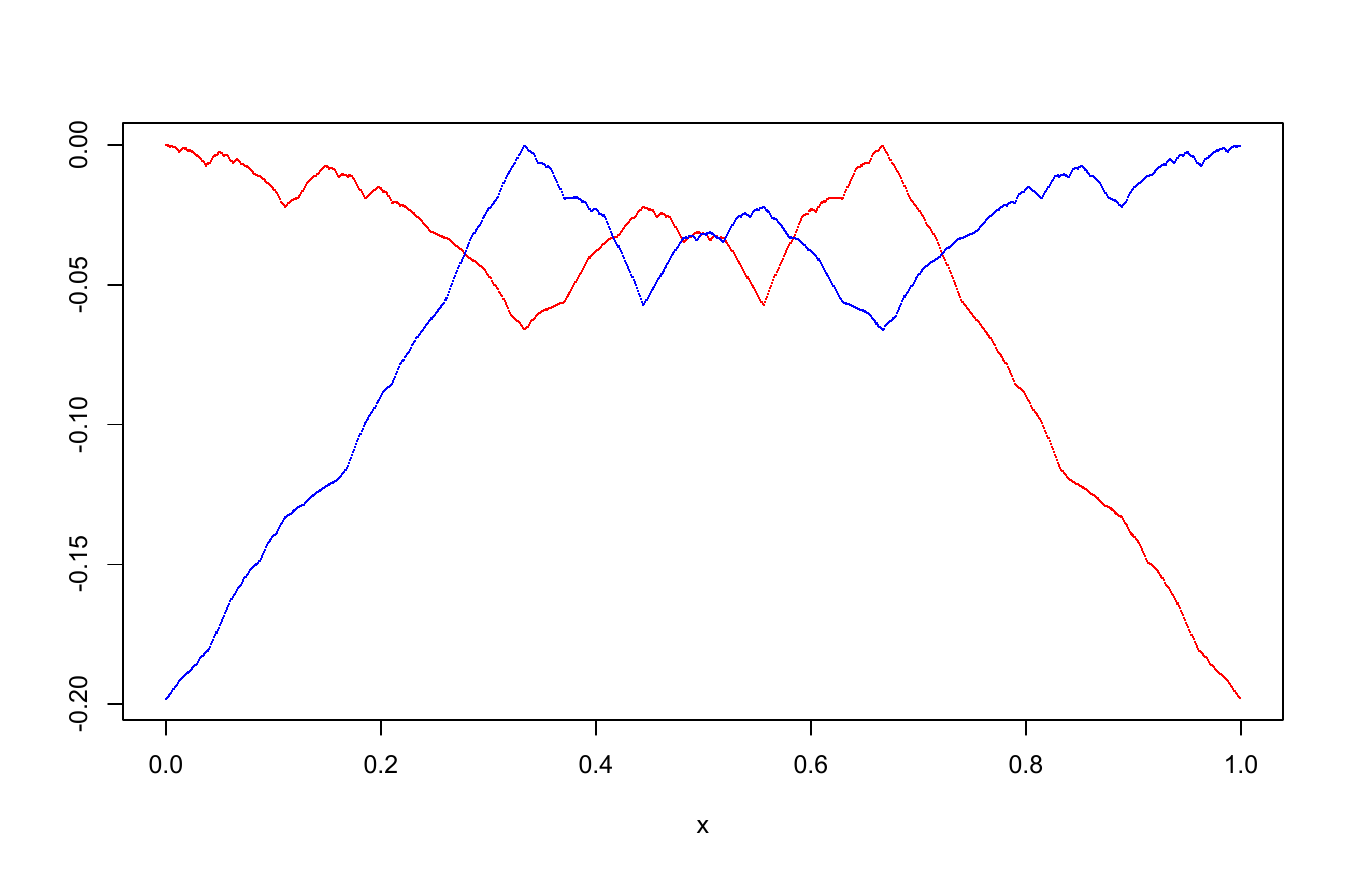}
    \caption{Picture obtained using the  $1/2$ iterative procedure for  $A_{100}(x)$ with a discretization of $10^5 $ points. The red graph shows when the realizer is obtained via $\tau_1$ and the one in blue is for the case when  the realizer is $\tau_2$.}
    \label{fig:my_label3}
\end{figure}
We point out that the above kind of reasoning is quite general; there are many similar examples: one can take another value of $t$, then, from the graph one gets from the  $1/2$ iterative procedure to guess the right number of $V_i$, etc.

\section{Revisiting the case   $A(x)=-(x-\frac{1}{3})^2$} \label{revis}

We can proceed in the same way as in the last examples by choosing a function $F$
and getting the power series for the case $A(x)=-(x-\frac{1}{3})^2$. We will get in the end the same result as in section \ref{pri}. We just outline the reasoning.
Taking
$F(x)=\frac{-21}{64}\left(x+1/9 \right)^2 +4/189,$
and
$\eta(x)=\tau_1 \circ \tau_1 \circ \tau_2(x),$
we will get

$\,\,\,\lim_{n \rightarrow + \infty}V_1^{n*}(x)=\frac{-21}{64}\sum_{i=0}^{+\infty} \left(\left( \eta^i(x)+1/9 \right)^2 -256/3969  \right).$

One can show that $\displaystyle \eta^i(x+1/7)=\frac{1}{7} +\frac{x}{8^i}.$
Therefore,
\begin{equation}\label{eq:serie_ex_compa}
V_1 (x+ 1/7)=\lim_{n \rightarrow + \infty}V_1^{n*}(x+1/7)= \frac{-21}{64}\sum_{i=0}^{+\infty}\left( \left(\frac{16}{63}+\frac{x}{8^i} \right)^2-\frac{256}{3969}\right).
\end{equation}
After simplification and canceling terms we get
$V_1(x)=-\frac{x^2}{3}-\frac{2x}{21}+1/49,$ which shows the same form (up to an additive constant) of the $V_1$ we obtained before on section \ref{pri}.
\begin{figure}[h]
  \centering
  \begin{minipage}[b]{0.4\textwidth}
    \includegraphics[height=3cm,width=\textwidth]{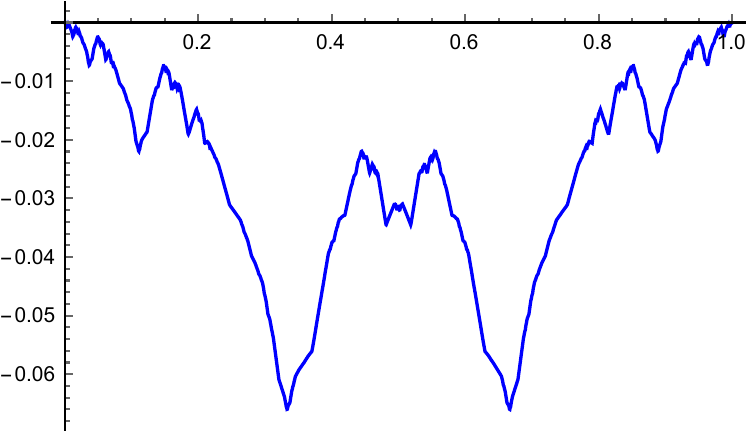}
    \caption{Truncation of the subaction $V$ as described on Theorem  \ref{Gfun}, where n=10.}
  \end{minipage}
  \hfill
  \begin{minipage}[b]{0.4\textwidth}
    \includegraphics[height=3cm,width=\textwidth]{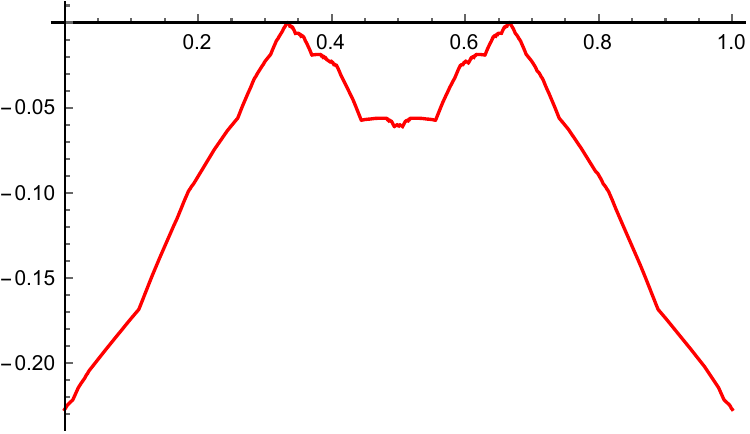}
    \caption{Truncation of the subaction $W$ as described on Theorem  \ref{Hfun}, where $n=10$.}
  \end{minipage}
\end{figure}
\begin{figure}[h]
    \centering
    \includegraphics[height=3cm,scale=0.4]{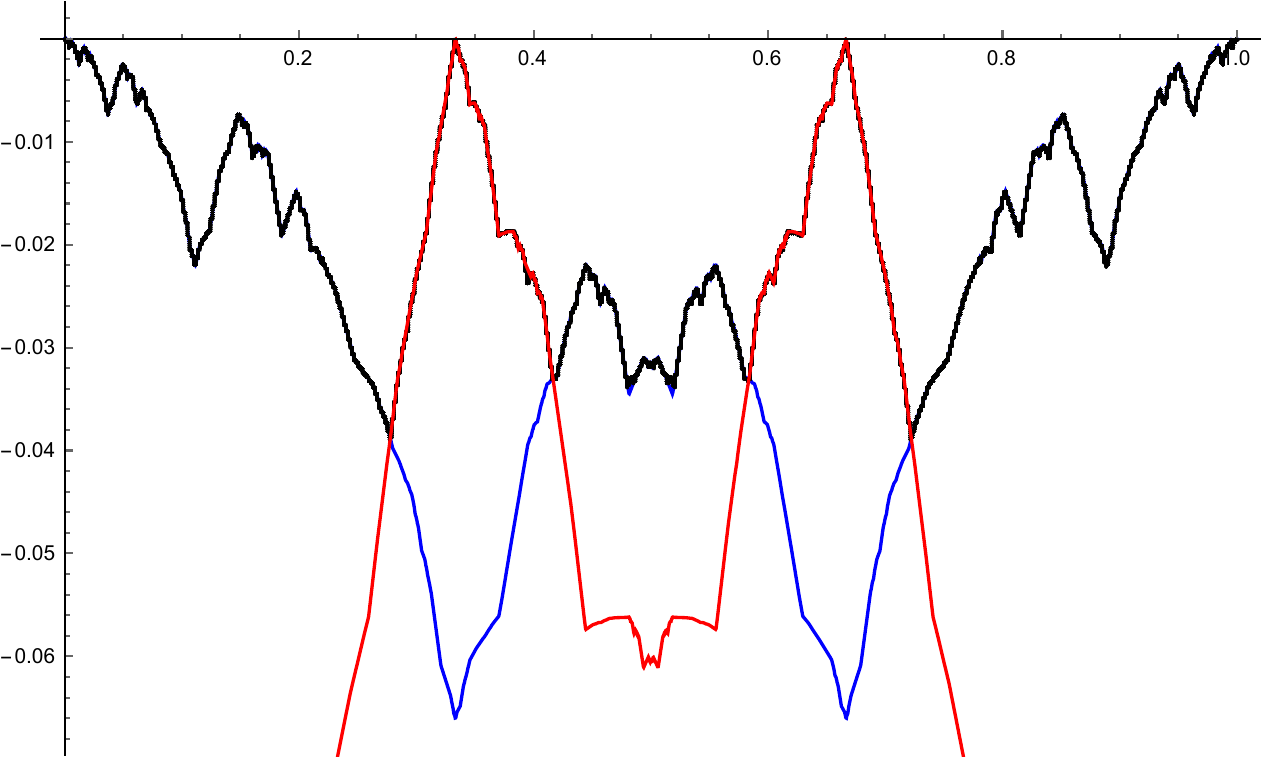}
    \caption{Superposition of the two above graphics resulting in a new subaction with the graph  in black. }
    \label{fig:my_label2}
\end{figure}

\section{Minus distance to the Cantor set} \label{Cantor}
Now, we  consider  the case where
$A(x)=-d(x,K)$ where \(d(x,K)=\min_{k\in K}|x-k|\) and $K\subset[0,1]$ is the Cantor.  Also, $T(x)= 2\,x$ (mod 1) acts on $[0,1]$ and the inverse branches of $T$ are  given by $\tau_1(x)=\frac{x}{2}$ and $\tau_2(x)=\frac{x+1}{2}$. In this section we present pictures we get from the use of the  $1/2$ iterative procedure and we present some conjectures.
We do not provide mathematical proofs. We  consider an approximation of the Cantor set via the mesh of points of the form
$m=\frac{1}{2}+\sum_{i=1}^{+\infty}a_i \frac{1}{3^i}$ where $a_i \in \{1,-1 \}$, and therefore we take
$\,\,\,\,\,\,A(x)=-d(x,K)=-\min_{(a_i) \in \{1,-1\}^\mathbb{N}}\left|x-\left( \frac{1}{2}+\sum_{i=1}^{+\infty}a_i \frac{1}{3^i}\right) \right|.$
It is easy to see that $m(A)=0$.
Note that $\{1/3, 2/3\,\}$ is contained on the Mather set. As $A$ is symmetric there is a symmetric subaction. Consider the truncation $\,\,\,\,\,\, A_n(x)=-\min_{(a_i)\in \{1,-1\}^{n}}\left|x-\left(\frac{1}{2}+\sum_{i=1}^{n}a_i \frac{1}{3^i}\right) \right|.$

The points $0$ and $1$ are also in the Mather set. We will try to get a subaction via
$V_1(x)-V_1(\tau_1(x))=A(\tau_1(x))$
and
$V_2(x)-V_2(\tau_2(x))= A(\tau_2(x))$. In this way we get
$V_1(x)=\sum_{i=1}^{+\infty}A\circ \tau_1^i(x)$ and $V_2(x)=\sum_{i=1}^{+\infty}A\circ \tau_2^i(x) = V_1(1-x).$
We conjecture that
$V(x)=V_1(x)I_{[0,1/2)}+V_1(1-x) I_{[1/2,1]}$
is a subaction
\begin{lemma}\label{lem_cantor1} The series
$G(x)=\sum_{i=1}^{+\infty} A(\tau_1^i(x))$ converges uniformly in $ [0,1]$.
\end{lemma}
\begin{proof}
Notice that \[G(x)= \sum_{i=1}^{+ \infty}-\min_{(a_i) \in \{1,-1\}^\mathbb{N}}\left|x/2^j-\left( \frac{1}{2}+\sum_{i=1}^{+\infty}a_i \frac{1}{3^i}\right) \right| \] and
\[\min_{(a_i) \in \{1,-1\}^\mathbb{N}}\left|x/2^j-\left( \frac{1}{2}+\sum_{i=1}^{+\infty}a_i \frac{1}{3^i}\right) \right| \leq \left|x/2^j \right|.\]
In this way  $|G|$  is bounded by a geometric series  and therefore we get the claim.
\end{proof}
As in the previous examples we also want to find $V,V_1,V_2$, such that,

$\max_{T(y)=x}[A(y)+V(y)]=\max\{V_1(x)+m(A),V_1(1-x)+m(A)\}=V(x).$

\begin{figure}[h]
    \centering
    \includegraphics[height=3cm,scale=0.45]{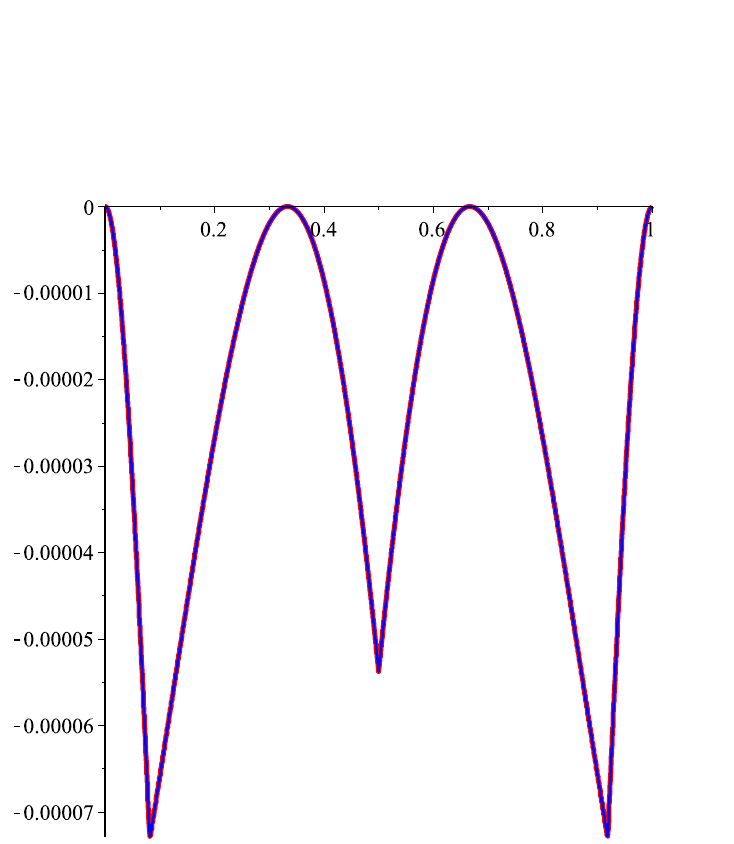}
    \caption{Case $A(x)=- x^2  (x-1/3)^2  (x-2/3)^2 (x-1)^2 $ -  This picture  describes the graph of the  function $V$ we get from the large iteration of $\mathcal{G}^{20}$ applied to the initial  function $f_0=0$. There is a numerical evidence that such $V$ is a calibrated subaction.}
    \label{fig:contraex_1}
\end{figure}
\begin{figure}[h]
    \centering
    \includegraphics[height=5cm,scale=0.55]{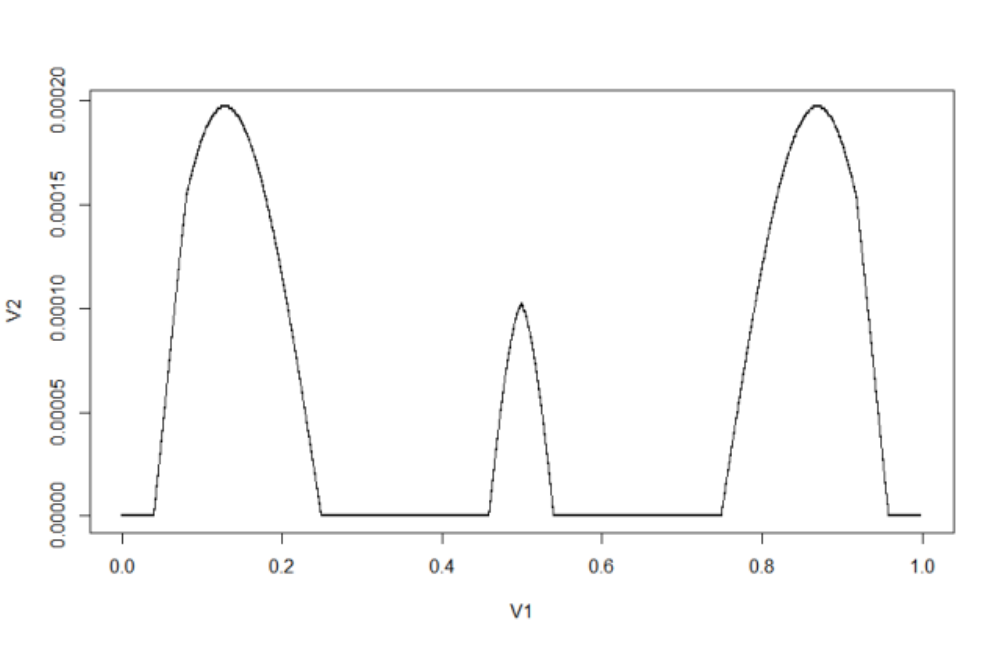}
    \caption{Case $A(x)=- x^2  (x-1/3)^2  (x-2/3)^2 (x-1)^2 $ -  This picture  describes the graph of the  function $R$ associated to the $V$ of last Figure  \ref{fig:contraex_1}. This graph confirm that there  is a numerical evidence that such $V$ is a calibrated subaction.}
    \label{fig:Rbom}
\end{figure}
\begin{conjecture}
\label{Gfun} Suppose $A(x)= -d(x,K)$ and $T(x)=2x$ mod(1), then, a subaction is given by
$V(x)=G(x)I_{[0,1/2)}(x)+G(1-x) I_{[1/2,1]}(x)$,
where
$G(x)=\sum_{i=1}^{+\infty} A(\tau_1^i(x)).$
\end{conjecture}
Now, we want to try to find  another subaction  but this time associated to  the maximizing probability with support on $\{1/3,2/3\}$. In this way we will look for solutions of the form
$V_2(x)=V_1(\tau_1(x))+A(\tau_1(x))$ and
$V_1(x)=V_2(\tau_2(x))+A(\tau_2(x))$.
As in the previous examples $\eta(x)= \tau_2(\tau_1(x))$ take
$V_2(x)= \sum_{i=1}^{+\infty} \left(A(\tau_1(\eta^i(x)))+A(\tau_2(\tau_1(\eta^i(x))) \right),$ and $V_1(1-x)=V_2(x).$ As $\eta(2/3)=2/3$ one can show that this series is absolutely convergent (similar to the previous Lemma \ref{lem_cantor1}). Define
$H(x)=\sum_{i=0}^{+\infty} \left(A(\tau_1(\eta^i(x)))+A(\tau_2(\tau_1(\eta^i(x))) \right)$. We want to show that
$W(x)=H(1-x)I_{[0,1/2)}(x)+H(x) I_{[1/2,1]}(x)$
is a subaction.  In the same way as before we want  to show that

$\,\,\,\,\max_{T(y)=x}[A(y)+V(y)]=\max\{H(x),H(1-x)\}=W(x).$

\begin{conjecture} \label{Hfun} The function $W$ given by
$\,\,\,\,W(x)=H(x)I_{[0,1/2)}(x)+H(1-x) I_{[1/2,1]}(x),$
$H(x)=\sum_{i=0}^{+\infty} \left(A(\tau_1(\eta^i(x)))+A(\tau_2(\tau_1(\eta^i(x))) \right),$
is a subaction for $A$.
\end{conjecture}
Above we conjectured that  $W $ and $ V$ were subactions. If this was true, then
$\max \{W+C_1,V+C_2\}$
is also a subaction, where $C_1,C_2\in \mathbb{R}$.

\section{A potential $A$ which is equal to its subaction $u$.} \label{equal}

Taking  $T(x)=2x$ mod(1) the inverse branches are $\tau_1(x)=1/2$, $\tau_2(x)=(x+1)/2$, which satisfy the equation $1-\tau_1(1-x)=\tau_2(x)$.
We will exhibit a potential $A$ which is equal to its subaction $u$. In order to derive the solution, we will make some assumptions on $u$.
 Suppose $u$ is symmetric of the form
\begin{equation}\label{u_A_potencial_fixo}
u(x)=\left\{\begin{matrix}
 f(x),&x<1/2
\\
f(1-x), & x\geq 1/2,
\end{matrix}\right.
\end{equation}
where
\[f(x)=\left\{\begin{matrix}
 g_1(x),& x<1/3
\\
g_2(x), &1/2 \geq  x\geq 1/3.
\end{matrix}\right.
\]
We assume that for  $x\in [0,1/2]$, the value $\displaystyle\max_{T(y)=x}[2\,u(y)]$ is realized by $\tau_2$ and for  $x\in [1/2,1]$ it is realized by $\tau_1$. Then, we get the system

$g_1(x)+{m}(A)=2g_2(1-\tau_2(x)),\,\, g_2(x)+{m}(A)=2g_1(1-\tau_2(x)),$

$g_1(1-x)+{m}(A)=2g_2(\tau_1(x)),\,\,\,g_2(1-x)+{m}(A)=2g_1(\tau_1(x)). $

Two of the above equations are redundant. Taking $\eta(x)=\frac{1+x}{4}$, we get the system
\begin{equation*}\label{eq_sis_fix} g_1(x)+{m}(u)=2g_2(1-\tau_2(x))\,\text{ and }\,g_2(x)+{m}(u)=2g_1(1-\tau_2(x)).
\end{equation*}

After some computations we  get
\[g_1(x)=\alpha \left(x-\frac{1}{3}\right) + \beta, \text{ and } g_2(x)=\alpha \left(\frac{1}{3}-x\right)+\beta,\]
with the constrains
$g_1(\tau_1(x))\leq g_2(1-\tau_2(x)),\quad x \in [0,1/3]$ and
$g_1(\tau_1(x)) \leq g_1(1-\tau_2(x)) \quad x \in [1/3,1/2].$ This means  (taking $\alpha>0$) that
$-\alpha/6 \leq 0, \quad x \in [0,1/3]$ and
$\alpha(x-1/2)\leq 0, \quad x \in [1/3,1/2].$
As in this case this is always true we finally get for $x \in (0,1/2)$, $\displaystyle\max_{T(y)=x}[2\,\,u(y)]=u(x)+\beta.$
By symmetry the same is true for $x \in (1/2,1).$
\begin{figure}[h]
    \centering
    \includegraphics[height=3cm,scale=1]{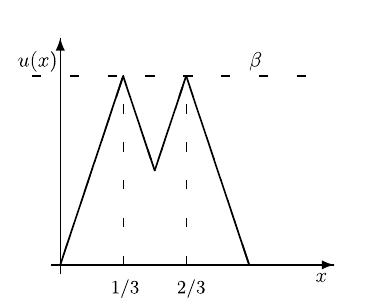}
    \caption{\textit{The graph of the function $A=u$ in  (\ref{u_A_potencial_fixo})}}
    \label{fig:my_label1}
\end{figure}
In this case $A=u$ (where $u$ is the subaction),  $m(A)=\beta$ and the maximizing probability has support on the orbit of period $2$. The general picture of the graph of $A=u$ is presented on Figure \ref{fig:my_label1}. A particular example could be $\alpha=0.4$ and $\beta=1$.

\section{Approximating the eigenfunction of the Ruelle operator} \label{Ru}

In this section, we will show that a variation of the  $1/2$ iterative procedure works fine also for approximating the eigenfunction of the Ruelle operator. Given a H\"older potential $A:[0,1] \to \mathbb{R}$ denote $L_A$ the Ruelle operator, that is given $f$, then $L_A(f)=g$, means
$$ L_A(f) (x)=g(x) = e^{A( \tau_1(x))} f( \tau_1(x) ) +  e^{A( \tau_2(x))} f( \tau_2(x) ).$$
It is known that there exists in this case an eigenvalue $\lambda>0$ and a positive eigenfunction $\varphi$ such that $L_A(\varphi)= \lambda \varphi$ (see \cite{PP}). We will define an operator $G$, such that, if $G(h)= h$, then, $e^h$ is the eigenfunction of the Ruelle operator.
\begin{figure}[h]
	\centering
	\includegraphics[height=3cm,scale=0.55]{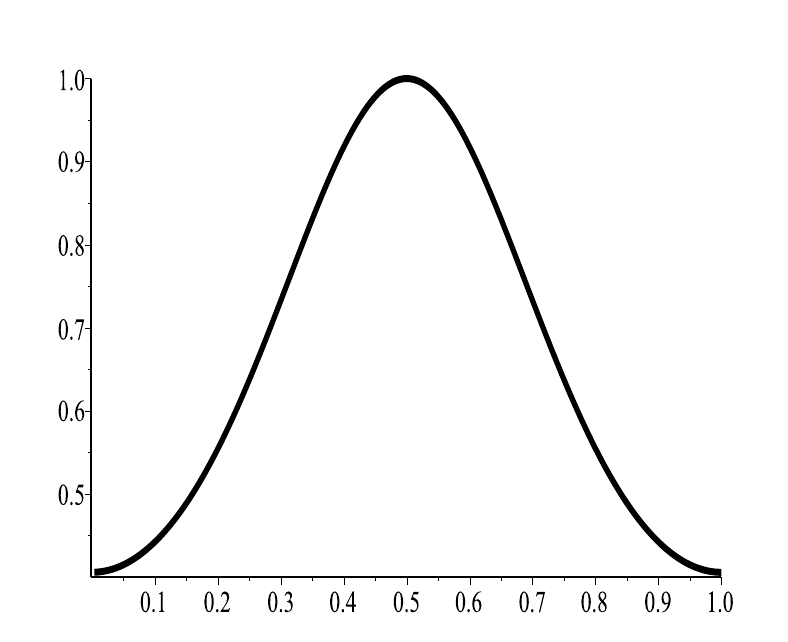}
	\caption{Case $A(x)= \sin^2 (2 \pi x)$ - approximating the eigenfunction and the eigenvalue  the Ruelle operator  - We consider the operator $G$ and the Ruelle operator $L_A$. In this Figure we plot the graph  of $ e^{G^{10}(0)}$  in blue and  the graph of
		$ \frac{1}{\lambda} \,L_A (e^{G^{10}(0)})$ in red, with $\lambda$ approximately equal to $3.472$.  There is a numerical evidence that $\varphi = e^{G^{10}(0)}$ is a good approximation to the eigenfunction of $L_A$. }
	\label{fig:fiu1}
\end{figure}
We define first the operator $\tilde{G}$ acting on functions in, such way that,
$$ \tilde{G}(g) = \frac{1}{2} g + \frac{1}{2} \log ( L_A (e^g)). $$
Finally, we define $G$ by $G(g) = \tilde{G} (g) - \tilde{G} (g)(0.5).$
One can show that for any $f,g>0$ we have that
$  |\tilde{G}(f) -    \tilde{G}(g) |_0\leq |f-g|_0$.

\medskip
\begin{remark}\label{oldremark4}
Once more the introduction of the $1/2$ factor helps on the procedure of iterating the operator $f \to G(f)$ to  an initial condition $f_0$. Indeed, in the same way as in Remark~\ref{oldremark3}, if for a point $z_0$ the signs of
$\frac{1}{2} (f-g) (z_0)$ and  $\frac{1}{2} \log ( L_A (e^g))(z_0)$ are different,
then one get a better contraction rate then one would get using the operator $f \to \log ( L_A (e^f))$.
\end{remark}
Suppose $G(h)=h$, then
\begin{equation} \label{cob}  h = \frac{1}{2} h + \frac{1}{2} \log ( L_A (e^h)) - c ,
\end{equation}
where $c$ is a constant.Take $\varphi$ such that $\log (\varphi)=h$ and $1/2\, \log (\lambda) =c.$ Then, we get
$  \frac{1}{2} \log (\varphi) = \frac{1}{2} \log ( L_A (\varphi)) - \frac{1}{2}\,\log \lambda .$
This means that
$ \lambda \,\varphi=  L_A (\varphi)) .$
We can approximate the eigenfunction $\varphi$ via high iterates of $G^n(0)$.
We applied this method for the potential $A(x)= \sin^2 (2 \, \pi \,x)$ and $T(x)= 2 x$ (mod 1).  Then, we plot
$e^{G^{10}(0)}$ and  $\frac{1}{\lambda}\,L_A (  e^{ G^{10}(0)}  )$ in Figure \ref{fig:fiu1} with $\lambda$ approximately equal to $3.472$.
We do not have to worry about the value $c$ above in equation \eqref{cob}. In order to estimate $\lambda$ we just take the value  $\lambda =\frac{L_A (  e^{ G^{10}(0)}  ) (0.4) }{e^{G^{10}(0)}(0.4)}.$

\section{The  $1/2$ iterative procedure applied to the   case where $A$ has more than one maximizing probability.} \label{more}

The discussion that will be made in this section only addresses questions regarding numerical evidence obtained from the  $1/2$ iterative procedure. We do not present rigorous proofs in this section. The interest here is to
understand better the dynamics of the  $1/2$ iterative procedure on the case where there is more than one maximizing probability. In this case  more than one calibrated subaction may exist. In some
sense, there are different basins of attractions for different subactions (depending where one begins - the initial condition -  the iteration of the  $1/2$ iterative procedure).

Consider the potential
$A(x)=- x^2  (x-1/3)^2  (x-2/3)^2 (x-1)^2 $ which has maximal value $m(A)=0$ and
Mather set equal to $\{0,1/3,2/3\}$ (when the setting is  $S^1$ and not $[0,1]$). The ergodic maximizing probabilities are $\mu_1=\delta_0$ and $\mu_2 = \frac{1}{2} (\delta_{1/3} + \delta_{2/3}).$
In this case, there exists more than one calibrated subaction (see Theorems 12 and 15 in \cite{GL1} or Theorem 5 in \cite{GLT}). One can get numerical evidence of the graph of these different calibrated subactions by considering the iteration of
$\mathcal{G}$ on distinct initial conditions. What kind of numerical evidence we can get from the use of the  $1/2$ iterative procedure? Taking the initial condition $f_0=0$ and iterating $\mathcal{G}$ we get the function $V$ which has  the graph shown on Figure \ref{fig:contraex_1}. This function $V$ "should be" a calibrated subaction. The graph of the associated function  $R$ (see expression (\ref{x1})) is displayed on Figure~\ref{fig:Rbom}. Suppose  we did not know
in advance where the Mather set is.  From Figure \ref{fig:Rbom} we have numerical evidence that the values of $R$ on the two periodic  orbits $\{0\}$ and $\{1/3,2/3\}$ are equal to zero (or, $\sim 0$).

The general idea is: even in the case the maximizing probability is not unique we get numerical evidence about the possible maximizing probabilities. Another initial condition $f_0$ can be attracted to another calibrated subaction $V$ by iteration of $\mathcal{G}.$ Indeed, let $\alpha_{\varepsilon,a}:[0,1] \to \mathbb{R}$ be a piecewise linear bump function defined by
\[ \alpha_{\varepsilon,a}(x)=
\left\{
\begin{array}{ll}
0, & 0 \leq x \leq a-\varepsilon \\
kx-k(a-\varepsilon), & a-\varepsilon \leq x \leq a \\
-kx+k(a+\varepsilon), & a\leq x \leq a+\varepsilon \\
0, & a+\varepsilon \leq x \leq 1 \\
\end{array}
\right.
\]
where $a \in (0,1)$ and $\varepsilon >0$ is arbitrary small. We consider two different initial conditions:\\
a) $A(x)=-x^2(x-1/3)^2(x-2/3)^2(x-1)^2$ and $f_0(x)=\alpha_{0.01,1/5}(x)$:
\begin{figure}[h!]
	\centering
	\includegraphics[height=3cm,width=4cm]{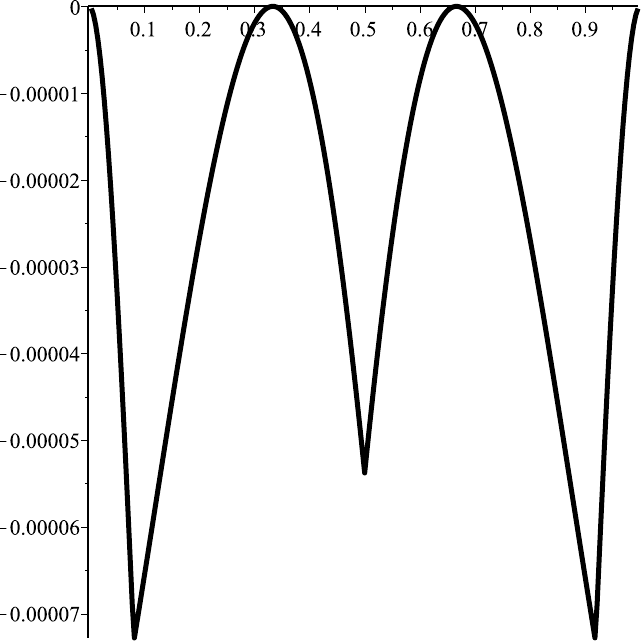}
	\caption{The approximated subaction obtained from the initial condition   $f_0(x)=\alpha_{0.01,1/5}(x)$.}\label{fig:elis1}
\end{figure}
In this case, there is numerical evidence that the high iterates  $\mathcal{G}^n (f_0)$ converge to  the graph described by Figure~\ref{fig:elis1}.\\
b)$A(x)=-x^2(x-1/3)^2(x-2/3)^2(x-1)^2$ and  $f_0(x)=\alpha_{0.01,2/3}(x)$:
\begin{figure}[h!]
	\centering
	\includegraphics[height=3cm,width=4cm]{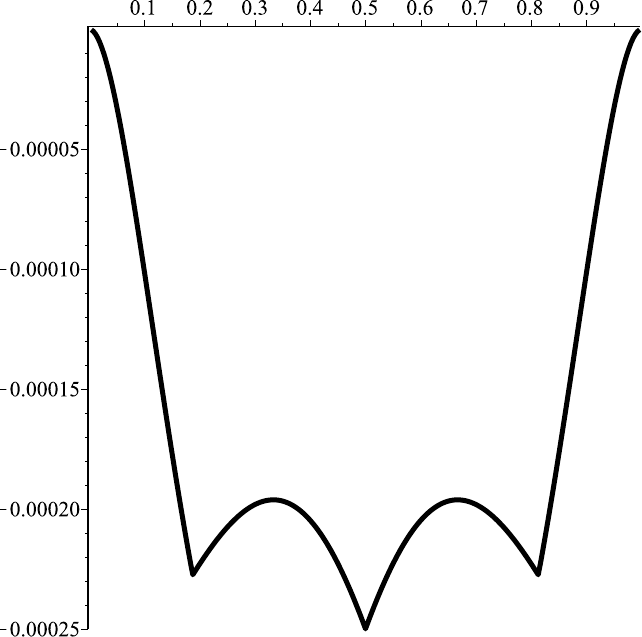}
	\caption{The approximated subaction obtained from the initial condition  $f_0(x)=\alpha_{0.01,2/3}(x)$.}\label{fig:elis2}
\end{figure}
In this case there is a numerical evidence that the high iterates of $\mathcal{G}^n (f_0)$ converge to  the graph described by Figure~\ref{fig:elis2}. In these two last cases, the graph of the corresponding $R$ (we do not present then here) also confirms the numerical evidence that such  functions $V$ are calibrated subactions. Interesting future work is to analyze the basin of attraction of each subaction by the iteration $\mathcal{G}^n$.
\begin{figure}[h]
	\centering
	\includegraphics[height=5cm,scale=0.85]{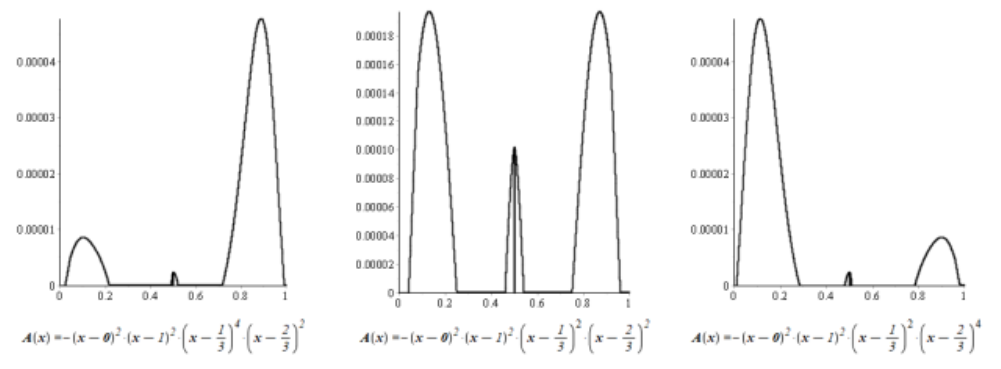}
	\caption{These pictures  describe the graph of the  function $R$ we get from the approximated calibrated subaction $V$ (obtained by large iteration of $\mathcal{G}$ applied to the initial  function $f_0= 0$) for different potentials $A$.}
	\label{fig:flatness_of_R}
\end{figure}
We can also show the influence of the flatness of the potential $A$ on the function $R$ (see Figure~\ref{fig:flatness_of_R}). One can see from these pictures  that the flatness  of the potential,  around a certain point on the Mather set, shows a clear influence on the size of the interval where $R\sim 0$ around this point. There is an increase of this size when the flatness increase. We strongly believe that the function $R$ is piecewise analytic (just proceeding in a similar way as in Section \ref{pri} or Section~\ref{revis}). Therefore, can not be equal to zero on an interval. The numerical roundoff error can cause a wrong impression (to be constant equal zero on an interval). The right conclusion is that the flater is the potential around one point in the Mather set more flat is $R$ around this point. The influence of flatness in the zero temperature limit of equilibrium probabilities was considered in \cite{BLL} and \cite{Mohr}.

\section{Appendix} \label{app}
\subsection{The subaction equation in the case     $A(x)=\sin^2(2 \pi x)$   } \label{ggs}
In this section we consider the case  $A(x)=\sin^2(2 \pi x)$ which was initially discussed on Section \ref{ggu}. We want to give more details on the  proofs. We want to show first that $ V(x)=\sup\,\{\,V_1(x),V_2(x)\,\} $
is a calibrated subaction for $A$, when $V_1$ and $V_2$ are described by (\ref{eq:sin2v}). Remember that for all $x$ we have $V_1(x)= V_2(1-x).$ Later we will present the  power expansion for $V_2$ which will show can be  described by (\ref{khj}).

\medskip

\begin{lemma}\label{lemma_maj} If $\displaystyle V_2(x)=\lim_{n \rightarrow +\infty}V_2^{n*}(x) $, then
$\displaystyle V_2(x)=\sum_{i=0}^{N}\left(F\circ \eta^i(x)-2\hat{m}(A)\right) + \epsilon_N(x)$,
where
$|\epsilon_N(x) | \leq 2\pi \sum_{i=N}^{+\infty}\frac{1}{4^i}= \frac{2 \pi}{3 \cdot 4^{N-1}} \leq \frac{2}{3 \cdot 4^{N-2}}.$
\end{lemma}
\begin{proof}
	 We just have to use the property that $\sin^2$ has Lipchitz constant equal $2$.
\end{proof}
We want to show that $V_2$ indeed satisfies \eqref{op2}.

\begin{lemma} \label{lemma_v_2ex1} If $\displaystyle V_2(x)=\lim_{n \rightarrow +\infty}V_2^{n*}(x) $, then
\[V_2(x)=V_2(\eta(x))+A\left(\frac{x}{2}\right)+ A\left(\frac{x}{4}+\frac{1}{2}\right)-2\,\hat{m}(A).\]
\end{lemma}
\begin{proof}
 Denote  $H(x)=A\left(\frac{x}{2}\right)+ A\left(\frac{x}{4}+\frac{1}{2}\right)-2\,\hat{m}(A)$.
Then, $\displaystyle V_2(x)=\sum_{i=0}^{+\infty}H(\eta^i(x))$
and $V_2(\eta(x))= \sum_{i=1}^{+\infty}H(\eta^i(x)).$
Therefore,
$V_2(\eta(x))=\sum_{i=0}^{+\infty}(H(\eta^i(x)) -H(x).$ From this follows
$V_2(\eta(x))=V_2(x)- H(x) ,$
and, finally
$V_2(x)=V_2(\eta(x))+A(\frac{x}{2})+ A(\frac{x}{4}+\frac{1}{2})-2\hat{m}(A)$.
\end{proof}
\begin{lemma}\label{lemma_v_2_ex_}
If  $\displaystyle V_2(x)=\lim_{n \rightarrow +\infty}V_2^{n*}(x) $ and $\hat{m}(A)=\frac{A(1/3)+A(2/3)}{2}$, then the function
$V_1(x)=V_2((x+1)/2)+A((x+1)/2)-\hat{m}(A)$
 satisfies
$V_1(x/2)+A(x/2)=V_2(x)+\hat{m}(A).$
\end{lemma}
\begin{proof}  From the relation between $V_1$ and $V_2$ we have
$V_2((x+1)/2)+A((x+1)/2)=V_1(x)+\hat{m}(A).$
Taking composition with $\tau_1(x)=x/2$ we get
\[V_1(x/2)+A(x/2)=V_2(x/4 + 1/2) +A(x/4 +1/2) + A(x/2) -\hat{m}(A)\]
\begin{equation} \label{eq:subst_ex2}
= V_2(\eta(x))+A(x/2) +A(x/4 +1/2) -\hat{m}(A).
\end{equation}
From Lemma \ref{lemma_v_2ex1} we obtain
$V_2(\eta(x))-V_2(x)=2\hat{m}(A)-(A(x/2) +A(x/4 +1/2)),$
therefore, adding and  subtrating $V_2(x)$ in (\ref{eq:subst_ex2}) we have
\[V_1(x/2)+A(x/2)=V_2(\eta(x))-V_2(x)+V_2(x)+A(x/2) +A(x/4 +1/2) -\hat{m}(A)\]
\[=2\hat{m}(A)-(A(x/2)+A(x/4+1/2))+V_2(x)+A(x/2) +A(x/4 +1/2) -\hat{m}(A).\]
Finally,
$V_1(x/2)+A(x/2)=V_2(x)+\hat{m}(A).$
\end{proof}

Now we need some differentiability results for  $V_1$ e $V_2$.

\begin{proposition} $V_2(x)$ is differentiable in $ [0,1]$ and $V_2'(x)=$

$\sum_{i=0}^{+\infty} 2 \pi(\eta^i)'(x)\left(\sin\left(\pi \eta^i(x) \right)\cos(\pi \eta ^i (x)) +\frac{1}{2}\sin \left( \frac{\pi \eta^i(x)}{2} \right)\cos \left (\frac{\pi \eta ^i (x)}{2} \right)  \right).$
\end{proposition}

We leave the proof for the reader.
\medskip

From the last proposition we get

$V_2'(x)=\sum_{i=0}^{+\infty} 2 \pi \frac{1}{4^i}\left(\sin\left(\pi \eta^i(x) \right)\cos(\pi \eta ^i (x)) +\frac{1}{2}\sin \left( \frac{\pi \eta^i(x)}{2} \right)\cos \left (\frac{\pi \eta ^i (x)}{2} \right)  \right).$

\begin{lemma} \label{trunc}
\(V_2'(x)=\varphi_N(x)+ \xi_N(x)\), where $|\xi_N(x)|\leq   3\pi \sum_{i=N}^{+\infty} |\frac{1}{4^i}|= \frac{\pi}{4^{N-1}},$

$\varphi_N(x)=\sum_{i=0}^{N} 2 \pi \frac{1}{4^i}\left(\sin\left(\pi \eta^i(x) \right)\cos(\pi \eta ^i (x)) +\frac{1}{2}\sin \left( \frac{\pi \eta^i(x)}{2} \right)\cos \left (\frac{\pi \eta ^i (x)}{2} \right)  \right).$
\end{lemma}

We leave the proof for the reader.

\medskip

$I_E$ denotes the indicator function of the interval $E$.

\begin{theorem} \label{tete} Taking  $\displaystyle V_2(x)=\lim_{n \rightarrow +\infty}V_2^{n*}(x) $ and $V_1(x)=V_2((x+1)/2)+A((x+1)/2)-\hat{m}(A)$, we get that
$V(x)=V_1(x)I_{[0,1/2)}(x)+ V_2(x) I_{[1/2,1]}(x).$
 is a calibrated subaction for $A$, when  $\hat{m}(A)=\frac{A(1/3)+A(2/3)}{2}=m(A)$.
\end{theorem}
\begin{proof} We have to show that
$
\max_{T(y)=x}[A(y)+V(y)]=\max \{V_1(x/2)+A(x/2),V_2((x+1)/2) + A((x+1)/2) \}.
$
As $V_1(u/2)+A(u/2)=V_2(u)+\hat{m}(A),$
and, \(V_1(x)=V_2(1-x)\), then, we have to show that
\begin{equation} \label{cort}
\max_{T(y)=x}[A(y)+V(y)]= \max \{V_2(x)+\hat{m}(A), V_2(1-x)+\hat{m}(A) \}
\end{equation}
We will show  first that if $u \in[0,1/2]$, then
$$V_2(u)+\hat{m}(A) \leq V_2(1-u)+\hat{m}(A)=V_1(u)+\hat{m}(A) . $$
Denote $\gamma(u)=V_2(u)-V_2(1-u)$.
By Lemma~\ref{trunc} we get
\[\gamma'(u)=V_2'(u)+V_2'(1-u)= \varphi_N(1-u)+\varphi_N(u) + (\xi_N(1-u)+\xi(u))\]
\[\geq \varphi_N(1-u)+\varphi_N(u) - 2\frac{\pi}{4^{N-1}}.\]
Taking $ N=4 $ it is easy to se that if $u \in [0.1,0.9]$  then $\gamma'(u)> 0$.
The function $\gamma$ is monotone increasing from $ 0.1$ to  $0.9$ and $\gamma(1/2)=0$.  Then $\gamma$ is negative on the interval $[0.1,0.5]$.
A similar argument can also handle the case  $x\in[0,0.1]$. We use Lemma  \ref{lemma_maj}, the fact that $\gamma(u)=V_2(u)-V_2(1-u)$ and the control of the error  $|\epsilon_N(x) |$.  Then,  finally we get that $\gamma$ is also negative
in $  [0,0.1]$  and is positive for  $x\in [0.9,1] $.
From the above we get
$\max_{T(y)=u}[A(y)+V(y)]= V_2(1-u)+\hat{m}(A),\quad u \in[0,1/2]$
and
$\max_{T(y)=u}[A(y)+V(y)]= V_2(u)+\hat{m}(A),\quad u \in[0,1/2].$
Therefore, for all  $x \in [0,1]$ we get
$\max_{T(y)=x}[A(y)+V(y)]= V(x)+ \hat{m}(A)$
Then,  $V $ is a calibrated subaction.
\end{proof}
Now we will express $V_2$ in power series. Our final result will be given by expression (\ref{power}).
Using the property \(\sin^2(x)=\frac{1-\cos(2\pi x)}{2}\),
we get
\begin{equation*}^{V_2(x+2/3)=\frac{1}{2}
\sum_{i=0}^{+\infty}\left( \sin\left(\frac{4\pi}{3}\right)\sin\left(2\pi\left(-\frac{1}{2}\right)^ix\right)- \cos\left(\frac{4\pi}{3}\right)(\,\cos\left(2\pi\left(-\frac{1}{2}\right)^i\,x\,-1\,)\right)\right)} .\end{equation*}

Now, define $$^{M(x)=\frac{\sin(4\pi/3)}{2}\sum_{i=0}^{+\infty}(\,\sin(2\pi(-1/2)^ix)\,-\, \sin(0)\,)}$$
and
$$^{Q(x)= \frac{-\cos(4\pi/3)}{2}\sum_{i=0}^{+\infty}(\,\cos(2\pi(-1/2)^i\,x\,-\, \cos(0)\,).} $$
We will express later $V_2$ as $V_2(x)=Q(x-2/3)+M(x-2/3).$
\begin{lemma} $ M$  and $ Q$ are uniformly convergent in each interval $[-a,a]$.
\end{lemma}
\begin{proof}  As the function $\sin$ is
Lipschitz, then, there is a constant $C $, such that,
\[|\sin(x)-\sin(y)|\leq C|x-y| \leq 2aC,\]
and
$\sum_{i=0}^{+\infty}\left|\sin\left(2\pi \left( -\frac{1}{2} \right)^i x\right) \right| \leq \sum_{i=0}^{+\infty}2\,a\,C\,\left|2\pi \left( -\frac{1}{2} \right)^i\right| \leq +\infty. $
For $Q$ we use an analogous argument.
\end{proof}
As
$\cos(x)=\sum_{k=0}^{+\infty}\frac{(-1)^{k}x^{2k}}{(2k)!}$
one can write $Q$  as
\begin{equation}\label{sin2:cosexp}
^{Q(x)=\frac{-\cos(4\pi/3)}{2} \sum_{k=1}^{+\infty} \sum_{i=0}^{+\infty}\left(\frac{(-1)^k(2\pi x)^{2k}}{2^{2ik}(2k)!}\right).}
\end{equation}
Finally, we get $Q(x)=\frac{-\cos(4\pi/3)}{2}\sum_{k=1}^{+\infty}\frac{(-1)^k(2\pi x)^{2k}}{(2k)!}\frac{2^{2k}}{2^{2k}-1}.$
Proceeding in analogous way we get
$M(x)=\frac{\sin(4\pi/3)}{2}\sum_{k=0}^{+\infty}\frac{(-1)^k(2\pi x)^{2k+1}}{(2k+1)!}\frac{2^{2k+1}}{2^{2k+1}+1}.$

\begin{proposition}
For a fixed $0<\varepsilon<1 $, if $x \in [-1+\varepsilon,1-\varepsilon]$, we can exchange the order in the sum of (\ref{sin2:cosexp}) and we  get
$$^{Q(x)=\frac{-\cos(4\pi/3)}{2}\sum_{k=1}^{+\infty}\frac{(-1)^k(2\pi x)^{2k}}{(2k)!}\frac{2^{2k}}{2^{2k}-1}.}$$
\end{proposition}
\begin{proof} Note that if $|x|<1$ there exists a constant $K$  (the coefficients on the power series of $\cos$ are decreasing) such that
$$^{\left|\sum_{k=1}^{+\infty}\frac{(-1)^k (2\pi x)^{2k}}{2^{2ik}(2k)!}\right| \leq  \sum_{k=1}^{+\infty}\left|\frac{(2\pi x)^{2k}}{2^{2ik}(2k)!}\right| \leq \frac{1}{2^i}\sum_{k=1}^{+ \infty}\left (K\, x^{2k} \right)=}$$
$$^{\frac{K}{2^i} \left( \frac{x^2}{1-x^2}\right) \leq \frac{K}{2^i}\left( \frac{|1-\varepsilon|^2}{1-|1-\varepsilon|^2}\right).}$$
We can exchange the order on the double sum: $\forall x \in  [-1+\varepsilon,1-\varepsilon]$,
$$^{\sum_{i=0}^{+ \infty}\sum_{k=1}^{+\infty}\left|\frac{(-1)^k (2\pi x)^{2k}}{2^{2ik}(2k)!}\right| \leq \sum_{i=0}^{+ \infty}\frac{K}{2^i} \left( \frac{x^2}{1-x^2} \right) \leq 2K\,\left( \frac{|1-\varepsilon|^2}{1-|1-\varepsilon|^2}\right) < +\infty.}$$
Note that $(x-2/3)\in[-2/3,1/3]$.
Then,
\begin{equation} \label{exp1} ^{Q(x-2/3)=\frac{-\cos(4\pi/3)}{2}\sum_{k=1}^{+\infty}\frac{(-1)^k(2\pi (x-2/3))^{2k}}{(2k)!}\frac{2^{2k}}{2^{2k}-1}.}\end{equation}
In the same way we get
\begin{equation} \label{exp2}
^{M(x-2/3)=\frac{\sin(4\pi/3)}{2}\sum_{k=0}^{+\infty}\frac{(-1)^k(2\pi (x-2/3))^{2k+1}}{(2k+1)!}\frac{2^{2k+1}}{2^{2k+1}+1}.}
\end{equation}
\end{proof}
As
$V_2(x+2/3)= M(x) + Q(x)$, then, $
V_2(x)=Q(x-2/3)+M(x-2/3).$
Finally, from (\ref{exp1}) and (\ref{exp2}) the power series expression of  $V_2$ around $2/3$ is given by
$$^{V_2(x)= \frac{\sin(4\pi/3)}{2}\sum_{k=0}^{+\infty}\frac{(-1)^k(2\pi  \left( x- \frac{2}{3} \right))^{2k+1}}{(2k+1)!} \frac{2^{2k+1}}{2^{2k+1}+1} -}$$
\begin{equation} \label{power}  ^{\frac{\cos(4\pi/3)}{2}\sum_{k=1}^{+\infty}\frac{(-1)^k(2\pi\left( x- \frac{2}{3} \right))^{2k}}{(2k)!}\frac{2^{2k}}{2^{2k}-1}}
\end{equation}

We can express the power series of $V_1$ around $1/3$ from $V_1(x)=V_2(1-x)$.

\subsection{The involution kernel for a map with a indifferent fixed point}  \label{fixind}

In this section, we show some results claimed on Section \ref{WKe}.

Consider $f:[0,1]\to [0,1]$, where
$$\left\{\begin{array}{l}
f(y)= \frac{y}{1-y},\mbox{ if },\, 0\leq y \leq \frac{1}{2},\\
f(y)= 2- \frac{1}{y},\mbox{ if }, \, \frac{1}{2}< y\leq 1,\\
\end{array}
\right.$$

and the potential $A(y) =  \log f '(y)$, which is given by the
expression
$$\left\{\begin{array}{l}
f'(y)= \frac{1}{(1-y)^2},\mbox{ if },\, 0\leq y \leq \frac{1}{2},\\
f'(y)= \frac{1}{y^2},\mbox{ if }, \, \frac{1}{2}< y\leq 1.\\
\end{array}
\right.$$
We want to derive the involution kernel for $A$. We claim the involution kernel for such $A$ is $ W(y,x) = 2 \log (x + y - 2 xy).$
We will show that

$\,\,\,\,\,\,\,\,\,A(F^{-1}(y,x))+    W( F^{-1}(y,x) )- W(y,x) =A(y).$

We denote $R_0\subset [0,1]^2$ the cylinder $0< y< 1/2$, and $R_1\subset [0,1]^2$ the cylinder $1/2< y< 1$. Restricted to $R_0$, the inverse $F^{-1}(y,x)$ is given by
$F^{-1}(x,y) =(\frac{y}{ 1-y}, \frac{x}{1+ x}).$ From this we get, for $(y,x) \in R_0$,
$ A(F^{-1}(y,x))= \log (1+x)^{-2}.$
Moreover, in this case, for $(y,x)$ in  the cylinder $R_0$,
$$^{ W( F^{-1}(y,x) )= 2 \,\log \left( \frac{y}{1-y} + \frac{x}{1+x} - 2\,\frac{\,y\, \,\,x}{(1-y)\,\,
(1+x)}\,\right)= \, 2 \log \left( \frac{x + y - 2 x\, y  }{(1-y)\,\, (1+x)} \right).}$$
Therefore, for $0<y<1/2$, we have
$$^{A(F^{-1}(y,x))+    W( F^{-1}(y,x) )- W(y,x) =}$$
$$  ^{\log (\, (1+x)^{-2}\,  \frac{( x+y - 2 x\, y)^{-2} }{(1-y)^{-2} \, (1 + x)^{-2}} \, \frac{1}{   ( x+y - 2 x\, y)^{-2}   } \,)=
 2\, \log (1-y)= A(y).} $$
Now we have to consider the cylinder $R_1$, where $ 1/2< y<1$. In this case,
$ F^{-1}(y,x)= ( 2- \frac{1}{y} , \frac{1}{2-x}).$
Therefore,
$F^{-1}(y,x)= \, 2 \, \log (2-x),$ and,
$$^{ W(   F^{-1}(y,x) )=2\,  \log(\,  \frac{2 y -1}{y} \,-\, \frac{1}{2-x} \, + \, 2 \frac{(2 \, y -1)}{y \, (2-x)}\,)=  }$$
$$^{2\, \log ( \, \frac{(2 y-1)\, (2-x) + y - 2\, (2y-1)}{y \, (2-x)}  \,)\,=\, 2\, \log( \,\frac{x+y - 2\, x\, y}{y \, (2-x)} \,)=}$$

Finally, for $1/2<y<1$, we have
$$^{A(F^{-1}(y,x))+    W( F^{-1}(y,x) )- W(y,x) =}$$
$$^{  \log (\, (2-x)^{-2}\,  \frac{( x+y - 2 x\, y)^{-2} }{y^{-2} \, (2- x)^{-2}} \, \frac{1}{   ( x+y - 2 x\, y)^{-2}   } \,)=
 2\, \log y= A(y).} $$
This shows that
$ W(y,x) = 2 \log (x + y - 2 xy)$
is the involution kernel for $\log f'(y)$.

\smallskip

We thank the referee for his careful reading which helped us to improve the reading of the text

\end{document}